%% file: BBMR-SEREVEM.tex
\documentclass{elsarticle}
\usepackage{amsmath,amssymb,amsfonts,graphics}

\usepackage{amssymb,graphicx,amscd,accents,array,url}
\usepackage[usenames,dvipsnames]{color}

\usepackage{amsmath}
\usepackage{latexsym}
\usepackage{mathrsfs}
\usepackage{amsthm}

\usepackage{epstopdf}

\usepackage{color}

\usepackage[english]{babel}
\numberwithin{equation}{section}

\allowdisplaybreaks

\newcommand\E{{E}}
\newcommand\PP{{\text P}}
\newcommand\Pt{{\rm P}}
\renewcommand\O{{\mathcal O}}

\newcommand{\dE}{\,{\rm d}\E}
\newcommand{\dP}{\,{\rm d}\PP}
\newcommand{\ds}{\,{\rm d}s}
\newcommand{\df}{\,{\rm d}f}

\newcommand\vphi{\varphi}

\newcommand\kD{k_{\Delta}}

\newcommand{\calD}{ \mathcal{D}}

\newcommand\Sp{{\mathscr S}}
\newcommand\R{\mathbb{R}}
\renewcommand{\P}{ {\mathbb P}}

\newtheorem{thm}{Theorem}[section]
\newtheorem{prop}[thm]{Proposition}

\newtheorem{remark}{Remark}

\newcommand{\p}{p}
\newcommand{\q}{q}
\newcommand{\pex}{\p_{\text{ex}}}
\newcommand{\bb}{\textbf{b}}
\newcommand{\diffp}{\kappa}
\newcommand{\reaction}{\gamma}
\newcommand{\aleQ}{\mathbb Q}
\newcommand{\aleS}{\mathbb S}

\title{Serendipity Nodal VEM spaces}

\author{
L. Beir\~ao da Veiga, F. Brezzi, L.D. Marini, A. Russo}
\address{{\it Louren\c co Beir\~ao da Veiga} -  Dipartimento di Matematica, Universit\`a di Milano Statale,
Via Saldini 50, I-20133 Milano (Italy),\\
lourenco.beirao@unimi.it}
\address{{\it Franco Brezzi} - IMATI del CNR,
Via Ferrata 5,
27100 Pavia, Italy,\\
brezzi@imati.cnr.it}

\address{{\it Luisa Donatella Marini} - Dipartimento di Matematica, Universit\`a di Pavia,
and IMATI del CNR, Via Ferrata 1, 27100 Pavia, Italy,\\
marini@imati.cnr.it}

\address{{\it Alessandro Russo} - Dipartimento di Matematica e Applicazioni, Universit\`a di Milano--Bicocca,
Via Cozzi 53, I-20153, Milano, Italy,\\
alessandro.russo@unimib.it}

\date{}

\begin{document}

\begin{abstract}
 We introduce a new variant of Nodal Virtual Element spaces that mimics the ``Serendipity Finite Element Methods'' (whose most popular example is the 8-node quadrilateral) and allows to reduce (often in a significant way) the number of internal degrees of freedom. When applied to the faces of a three-dimensional decomposition, this allows a reduction in the number of {\it face} degrees of freedom: an improvement that cannot be achieved by a simple static condensation. On triangular and tetrahedral decompositions the new elements (contrary to the original VEMs) reduce exactly to the classical Lagrange FEM. On quadrilaterals and hexahedra the new elements are quite similar (and have the same amount of degrees of freedom) to the Serendipity Finite Elements, but are {\it much more robust} with respect to element distortions. On more general 
 polytopes the Serendipity  VEMs are the natural (and simple) generalization 
 of the simplicial case.
\end{abstract}
\bigskip
\maketitle

\section{Introduction}

The original Virtual Element Methods, as introduced in \cite{volley}, show a surprising robustness with respect to the variety of shapes
allowed for the geometry of elements, and compared to Finite Elements allow a much easier construction of $C^1$ elements (and actually also $C^2$ or more). These aspects raised the interest of several  groups working on various applications (as for instance
topology optimization in elasticity problems {\cite{VEM-topopt}},  fractured materials \cite{Berrone-VEM}, plate bending problems \cite{Brezzi:Marini:plates}, or the Cahn-Hilliard equation \cite{ABSVu}).

An interesting feature is surely the possibility of joining classical Finite Elements (on rectangles or quadrilaterals) in some part of the domain, and VEMs
in some other part, as the two methods share the same trial functions and degrees of freedom on edges.  But as far as the {\it internal } degrees of
freedom are concerned, on simple geometries, as on triangles, VEMs are more expensive than the traditional Finite Elements: for a given accuracy $k$,  VEMs on triangles  use (together with polynomials of degree $k $ on every edge) a number of internal degrees of freedom equal to   $k(k-1)/2$, instead of the 
$(k-1)(k-2)/2$ used by Finite Elements. This would also imply that the possibility of combining FEM and VEM is not immediate in three dimensions even when the common face is a triangle. 


On quadrilaterals,  VEMs have again $k(k-1)/2$ internal degrees of freedom, that now should be  compared to
the $(k-1)^2$  internal degrees of freedom of ${\mathbb Q}_k$-Finite Flements, or to the
$(k-2)(k-3)/2$ internal d.o.f.s of Serendipity FEM (on quadrilaterals).

However, on non-affine quadrilaterals the Serendipity Finite Elements suffer a severe deterioration of accuracy: see e.g. \cite{A-B-F-nod} or the more recent \cite{A-A-Sere},  \cite{Gillette-14}. See also \cite{Periodic-SIAM} for a general survey on the various Finite Element choices. On the contrary, VEMs have in their {\it robustness with respect to distortion} one of their most relevant advantages.

On the other hand, the biggest advantage of classical FEM  (over Virtual Elements and similar methods)  is surely  the fact that  the values of trial or test functions of FEMs can be easily computed at any point, while VEMs are easily computed only along the edges. The common remedy, for VEMs, is to use (for computing point values and for similar information), instead of the true trial and test functions, their $L^2$-projection on some polynomial space of degree, say, $r$. For the original VEMs in \cite{volley} we could take only $r=k-2$ (with an obvious lack of accuracy) or use other, non orthogonal, projectors (a procedure that needed a theoretical justification).  However, for their advanced versions, as in \cite{projectors}, we could reach $r=k$ still using $k(k-1)/2$ internal degrees of freedom. This however, on simple elements like triangles or tetrahedra, was still higher than the FEM counterpart.

Here we propose a variant of VEMs that mimics, in some sense, the Serendipity approach of FEMs. The new variant 
coincides exactly, on triangles, with traditional Finite Elements, and on quadrilaterals can (in some sense) keep all the good aspects of Finite Elements without most drawbacks. In particular, on parallelograms we use $(k-2)(k-3)/2$ degrees of freedom (as for Serendipity FEMs) and we can easily compute
the $L^2$-projection on $\P_k$, but we can also keep the same accuracy when the element is strongly distorted. The only degeneracy that is not fully allowed is when the quadrilateral element becomes a triangle (as in the second element of Figure \ref{dege} below). But in that case (even in the limit,
when the quadrilateral is {\it exactly} a triangle), we can keep optimal accuracy just by using $(k-1)(k-2)/2$ degrees of freedom
(the same amount that we would use on a triangle). 

Moreover, the edge degrees of freedom  are exactly the same as for finite elements, so that in 2 dimensions we can combine
the two methods (using each in a different part of the domain). The same is now true also in three dimensions, if the matching VEM-FEM
is done on triangular faces, and even the matching through quadrilateral faces could be easily arranged (for instance with a slightly nonconforming 
matching).

Our construction is a mixture of Serendipity ideas and of the ones coming from enhanced elements of \cite{projectors}. Roughly speaking, instead
of keeping (among the degrees of freedom) the moments up to the order $k-2$ (as in the original VEMs), we go down to $k-3$,
and we use the boundary d.o.f.s and the internal moments up to $k-3$ to compute a {\it rough projector} from the VEM space onto $\P_k$. Then we use 
such a {\it rough projector} to define the moments of degree up to $k$ as a byproduct.

Throughout the paper we will use the following notation.

\medskip


For $k\ge 0$ and $d\ge 1$ integer we denote by
$\P_{k,d}$ the set of polynomials of degree $\le k$ in $d$ variables.
 Often, the dimension $d$ will be omitted when it is reasonably clear form the context. With a (rather common) abuse of language we also set $\P_{-1}\equiv\{0\}$. Whenever convenient, for a generic geometric object $\O$ in $d$ dimensions we will denote by $\P_{k,d}(\O)$ the restriction to $\O$ of $\P_{k,d}$. 
 

Following \cite{super-misti} we denote by $\pi_{k,d}$ the dimension of $\P_{k,d}$  (that is, for instance, $(k+1)(k+2)/2$ in two variables and  
$(k+1)(k+2)(k+3)/6$ in three variables).

An outline of the paper is as follows. In Section 2 we recall the original VEMs in 2 dimensions, and we compare them with classical Lagrange Finite Elements on triangles, and with classical $\mathbb{Q}_k$ and Serendipity Finite Elements on parallelograms and
quadrilaterals. In Section 3 we introduce our new Serendipity Virtual Elements in 2 dimensions, and we extend them to the three dimensional case in Section 4. Numerical examples involving the convection-diffusion-reaction equation are presented in Section 5.

\section{Original Nodal VEMs}

\subsection{Original nodal Elements in 2 dimensions}

Here below we recall the original ``nodal Virtual Element'' as reported in \cite{volley, hitchhikers} for the two dimensional case, and in \cite{projectors} for the three-dimensional one. As common, we will
concentrate on the description of the finite dimensional spaces within a single element (polygon) $\E$. The
assembling of the spaces on the whole computational domain will then be done with the same procedure
followed for $H^1$-conforming Finite Elements.


As is already well known, Virtual Elements allow an enormous generality in the geometry
of the elements to be used in the decomposition of the computational domain, and the precise limits of this generality are, in some cases, still to be understood.  For simplicity, here we will consider the typical assumption (see for instance \cite{volley}): there exists a fixed number $\rho_0>0$, independent of the decomposition, such that for every element $\E$ (with diameter $h_\E$) we have that: {\it i)} $\E$ is star-shaped with respect of all the points of a ball of radius $\rho_0\,h_\E$, and {\it ii)} every edge $e$ of $\E$ has length $|e|\ge \rho_0\,h_\E$. Actually, more general assumptions could be allowed in the definition of our VEM
spaces, but this goes beyond the scope of the present paper (again, see for instance \cite{volley}). Figure \ref{general} will show some examples of polygons that 
are indeed acceptable for our constructions. 

\begin{figure}[htbp]
  \begin{center}
    \includegraphics[width=10.7cm]{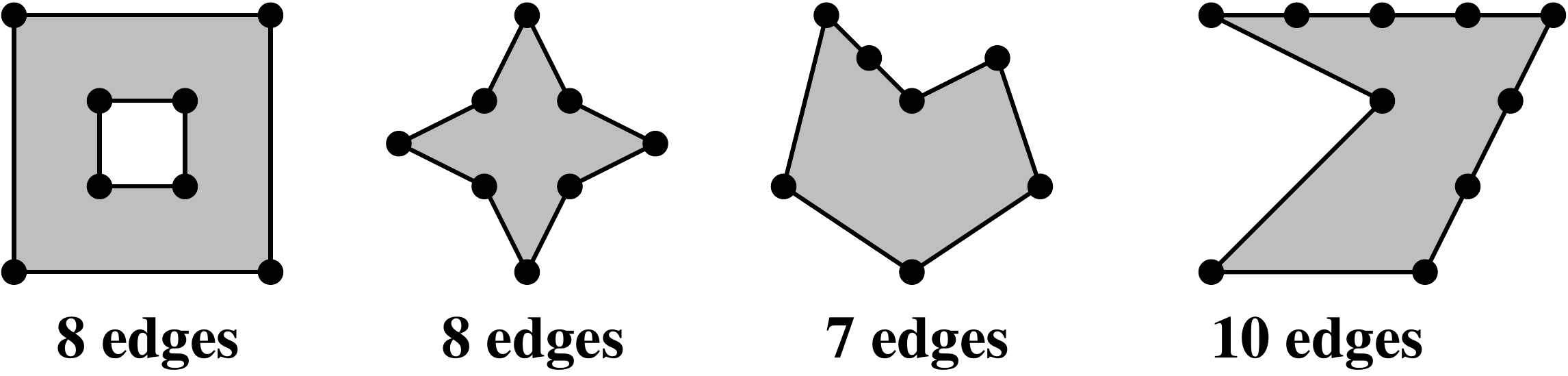}
  \end{center}
  {\caption{Element shapes allowed in our construction}\label{general}}
\end{figure}

For $k$ integer $\ge 1$  we define
\begin{equation}\label{deforig}
V_{k}(\E) =\{\varphi\in C^0(\overline{\E}):\,\varphi_{|e}\in\P_k(e)\mbox{ for all edge } e, \mbox{ and }\Delta\varphi\in
\P_{k-2}(E)\} .
\end{equation}
The degrees of freedom in $V_{k}(\E)$ are taken as
\begin{equation}\label{dofn-0}
\bullet\;\,\mbox{the values of $\varphi$ at the vertices},
\end{equation}
\begin{equation}\label{dofn-1}
\hskip1.8truecm\bullet  \int_e\varphi\, q\ds\quad \mbox{ for all }q\in\P_{k-2}(e)\quad\forall \mbox{ edge } e ,
\end{equation}
\begin{equation}\label{dofn-2}
\hskip0.2truecm\bullet \int_\E \varphi\, q\dE\quad \mbox{ for all }q\in\P_{k-2}(\E).
\end{equation}
It is immediate to verify that the degrees of freedom \eqref{dofn-0}-\eqref{dofn-2}
are unisolvent  (see \cite{volley}). For convenience of the reader we recall the proof. The number of degrees of freedom
in \eqref{dofn-0}-\eqref{dofn-2} is obviously equal to the dimension of  the space $V_k(\E)$ in
\eqref{deforig}: for a polygon of $N_e$ edges, they are both equal to $k\,N_e$ (number of boundary d.o.f.s) plus $\pi_{k-2,2}$ (dimension of $\P_{k-2}$ in two variables). Assume now that for a given
$\vphi\in V_k(\E)$ we have that all \eqref{dofn-0}-\eqref{dofn-2} are identically zero. Then
clearly $\vphi$ would be zero on the boundary (from \eqref{dofn-0}-\eqref{dofn-1}) and then using
\eqref{dofn-2} we would have $\int_{\E}|\nabla\vphi|^2\dE=-\int_{\E}\vphi\Delta\vphi\dE=0$
since $\Delta\vphi$ is a polynomial of degree $k-2$. This ends the proof.

The spaces $V_k(\E)$ are, in some sense, the {\it basic ones} in the VEM theory, similarly to, say, Lagrange finite elements on triangles for the FEM theory. However, compared with FEM (on triangles and on quadrilaterals) they show some differences, already in the number of {\it internal} degrees of freedom.

Comparing these (original) VEMs with the classical Finite Elements, whenever possible (meaning, here, for
triangular or quadrilateral elements) we find that on the boundary of the elements we have (or we can easily take) 
the same degrees of freedom. In the interior, however, this is not the case. In particular, on triangles, Virtual Elements have {\it more} degrees of freedom than the corresponding Finite Elements, and more precisely: the number of internal
degrees of freedom for Virtual Elements of degree $k$ is equal to $\pi_{k-2,2}$ while that of the corresponding Finite Elements is $\pi_{k-3,2}$ (see Figure \ref{clastri}).  For quadrilaterals, instead,  the number of internal nodes
for Finite Elements is equal to the dimension of ${\mathbb Q}_{k-2}$ (which is $(k-1)^2$),  while  for Virtual Elements the internal degrees of freedom equal the dimension of $\P_{k-2}$ (that is $k(k-1)/2$). See Figure \ref{clasqua}.

\begin{figure}[htbp]
  \begin{center}
    \includegraphics[width=7.7cm]{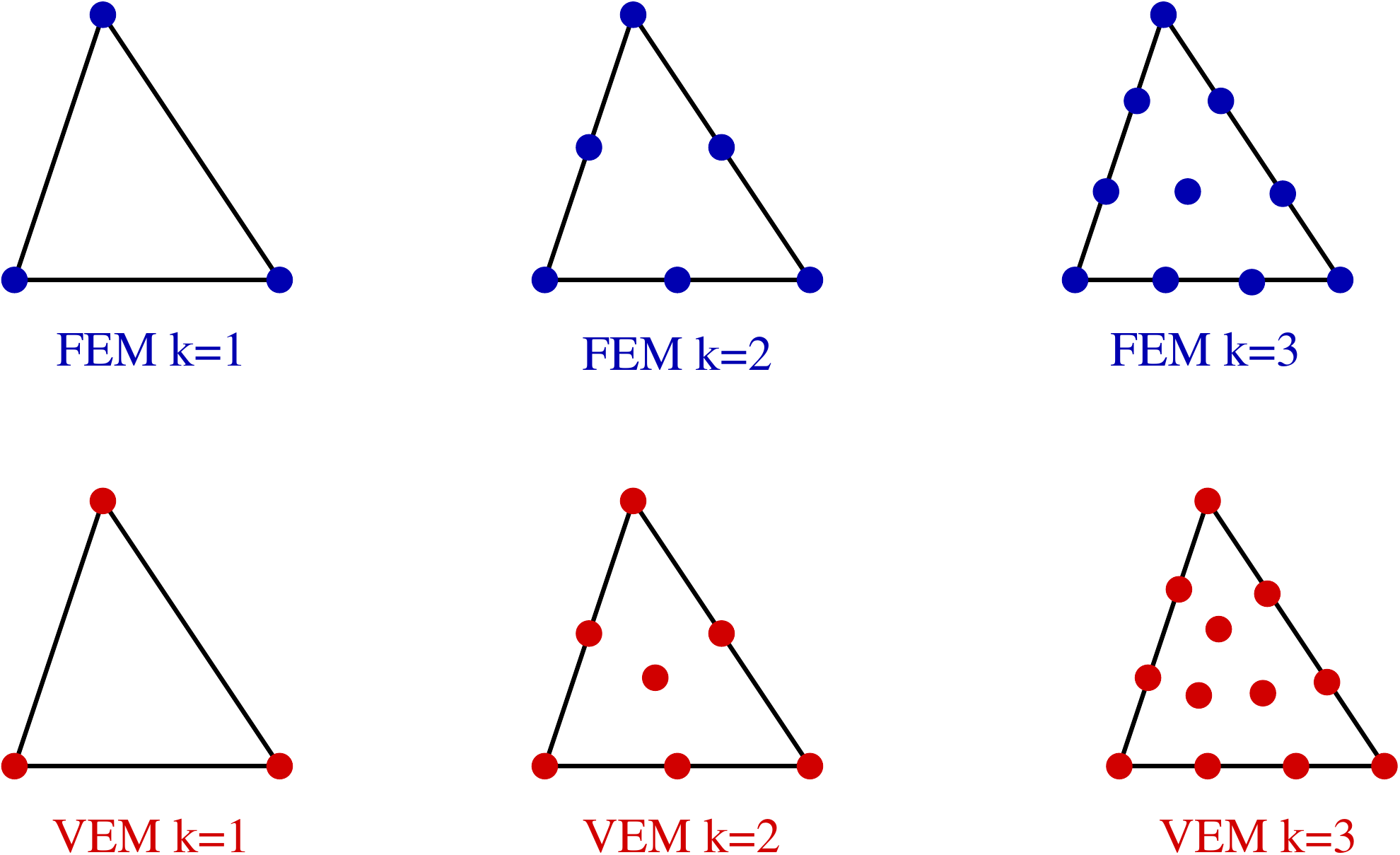}
  \end{center}
  {\caption{Triangles: Classical FEM and Original VEM}\label{clastri}}
\end{figure}

\begin{figure}[htbp]
  \begin{center}
    \includegraphics[width=7.7cm]{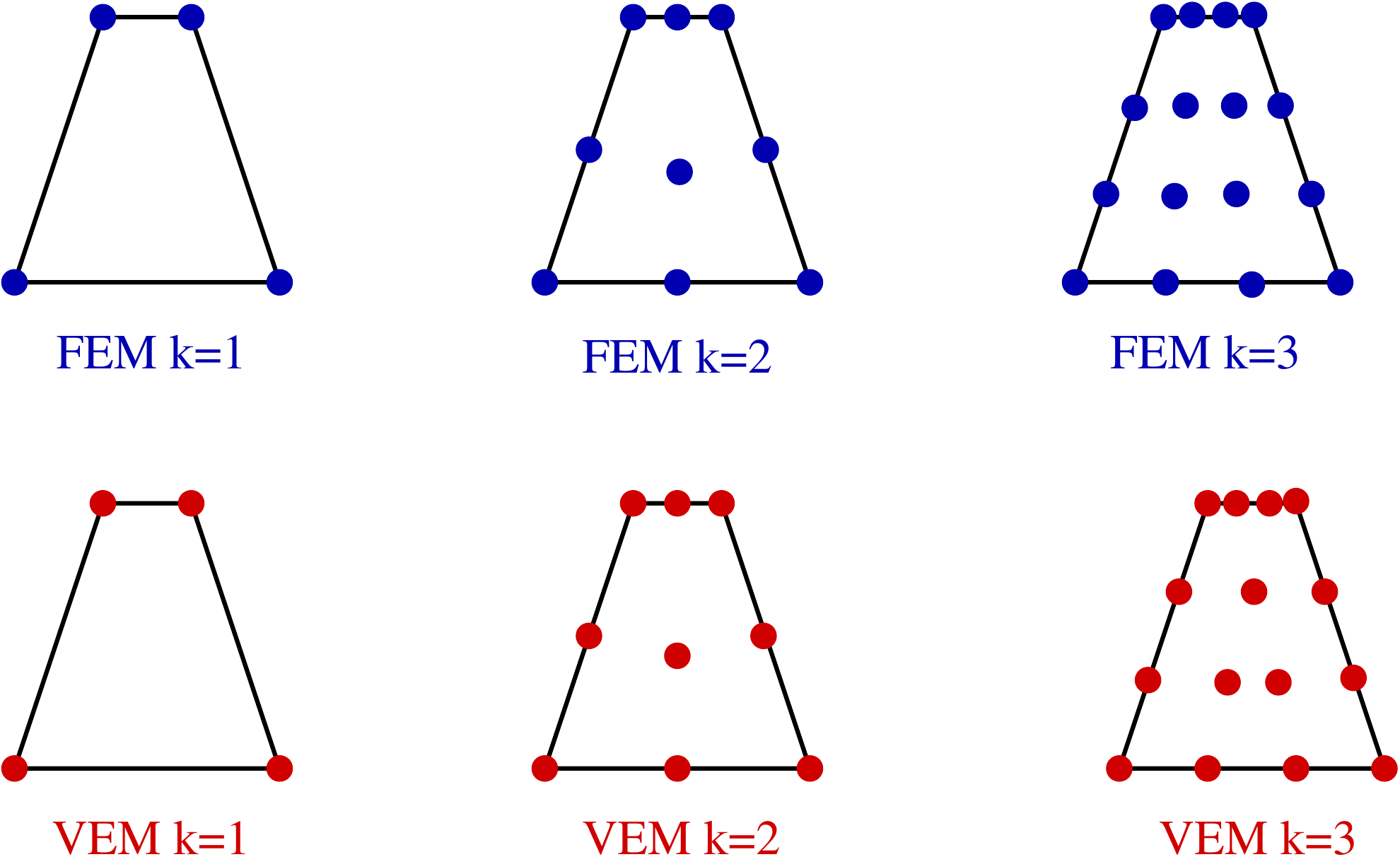}
  \end{center}
  {\caption{Quads: Classical ${\mathbb  Q}_k$-FEM and Original VEM}\label{clasqua}}
\end{figure}

\begin{remark} As we already mentioned, for the present 2-dimensional case  the restriction to each edge of Finite Elements
and of Virtual Elements is the same (both being polynomials of degree $\le k$ in
one dimension), so that we could actually allow a combined use of traditional Finite Elements (in some parts of the
computational domain) and of Virtual Elements (in other parts). 
\end{remark}

\begin{remark} In addition to the previous remark, we observe that for Virtual Elements we could very easily consider cases in which different degrees are used (say, in \eqref{dofn-1}) for different edges of the same polygon. In this case we note that: i) the order of accuracy on every polygon will be reduced to the lowest among the orders of the single edges, and
ii) in the global setting, to ensure conformity, the degrees of freedom on an edge shared by two polygons must obviously be the same. 
{This generalization could be, for instance, useful to develop $hp$ Virtual Elements in a very natural way.}

\end{remark}

Needless to say, the number of degrees of freedom for a given accuracy is not, by far, the whole story. One has to see what should be done with them;
  but this goes beyond the aims of the present paper.

\subsection{More general nodal VEMs}

For integers $k$  $\ge 1$ and $\kD \ge -1$ we define
\begin{equation}
V_{k,\kD}(\E) =\{\varphi\in C^0(\overline{\E}) :\,\varphi_{|e}\in\P_k(e)\,\forall\mbox{ edge } e, \mbox{ and }\Delta\varphi\in
\P_{\kD}(E)\} .
\end{equation}
The degrees of freedom in $V_{k,\kD}(\E)$ are taken as
\begin{equation}\label{dofn-0g}
\bullet\;\,\mbox{the values of $\varphi$ at the vertices} ,
\end{equation}
\begin{equation}\label{dofn-1g}
\hskip-0.3truecm\bullet \int_e\varphi\, q\ds\quad \mbox{ for all }q\in\P_{k-2}(e) ,
\end{equation}
\begin{equation}\label{dofn-2g}
\bullet \int_\E \varphi\, q\dE\quad \mbox{ for all }q\in\P_{\kD}(\E) .
\end{equation}
Clearly, the previous case \eqref{deforig} corresponds to the choice $\kD=k-2$. 

The extension of the previous unisolvence proof to the more general case of the degrees of freedom \eqref{dofn-0g}-\eqref{dofn-2g}  is an exercise. We also point out that, for $\kD\ge 0$ the degrees of freedom \eqref{dofn-2g} allow for the computation of the $L^2$-orthogonal projection operator $\Pi^0_{\kD}$, from $V_{k,\kD}(\E)$ to $\P_{\kD}(\E)$. As we shall see, the possibility
to compute this operator with an algorithm that uses only the degrees of freedom is one of the crucial steps in Virtual Element Methods.

We remark that the space $V_{k,\kD}(\E)$ clearly contains the space of polynomials $\P_s(\E)$ for
all $s\le {\min}\{k,\kD+2\}$, but $\Pi^0_r$ can be computed (out of the degrees of freedom), only
for $r\le \kD$. 

It is also clear that a smaller $\kD$ will correspond to a smaller number of degrees of freedom. However, as we have seen, for $\kD<k-2$ the space $V_{k,\kD}$ will fail to contain all polynomials of $\P_{k}$.

On the other hand, the choice $\kD=k$ would allow an immediate computation of the moments up to the order $k$, and hence the computation of the $L^2$-projection operator $\Pi^0_k$ that, as we said, is extremely useful. But for $\kD=k$ the degrees of freedom \eqref{dofn-2g} would
be very expensive.

Nevertheless, looking at Figure \ref{clastri}, we {\it feel} that there should be something better that can be done. To explain it, we start with some simple observations on polynomials that vanish on the boundary of a polygon. 

\subsection{Polynomials that vanish on $\partial\E$}

We start by noting that: {\it If a polynomial $p_k(x,y)$ of degree $\le k$ vanishes identically on a segment
(of positive length) that belongs to the straight line with equation, say, $ax+by+c=0$, then $p_k$ can be
written as $p_k=(ax+by+c)\,q_{k-1}$ with $q_{k-1}$ a polynomial of degree $\le k-1$}.  
The property is very well known, but if one needs more details we refer, for instance, to
 Lemma 3.1.10 of \cite{Brenner:Scott}.
 
As a consequence, a polynomial that vanishes identically on $\partial\E$ will contain, in its expression, the product of all the different straight lines that contain at least one edge of $\partial\E$. Note that even if several edges belong to the same line, (see for instance the fourth case in Figure \ref{general}) the equation of the line will always appear {\it once} (and not as many times as there are edges). For instance, looking again at the fourth case of Figure \ref{general}, we have ten edges but  we have to count only  five lines.

In general, given a polygon $\E$, we will denote by  $\eta_\E$ the number of distinct straight lines that contain at least one edge of $\E$. This is an important notation, that deserves to be better highlighted:
\begin{equation}\label{defeta}
\eta_\E=\mbox{ minimum number of straight lines needed to cover all } \partial\E.
\end{equation}

Having said that, we note that  for every
$k<\eta_E$ we obviously have
\begin{equation}\label{polynull}
\forall p_k\in\P_{k,2} \qquad \{p_k=0\mbox{ on }\partial\E\}\Longrightarrow \{p_k\equiv 0\}.
\end{equation}

With this, and noting that for every polygon $\E$ we always have $\eta_E\ge 3$,  it is not difficult to see that, for instance, a polynomial of degree $k \le 2$ is uniquely identified by its values at the boundary of any polygonal element $\E$.
 As a consequence, knowing the boundary value of a polynomial of degree $\le 2$ we know the whole polynomial, and hence we know its mean value (and, if needed, its moments of any degree). Why should we need
internal degrees of freedom?  

More generally, for $k\ge 3$ on triangles it is easy to see (looking for instance at the classical 
Finite Elements, see again Figure \ref{clastri})  that a polynomial of degree $\le k$ is uniquely identified by its boundary values and by its moments of degree $\le k-3$, and we shouldn't need  the moments of degree $k-2$. And on a more general polygon $\E$, with $\eta_\E>3$, the boundary values should count even more. So why should we need the moments of degree $k-2$?

A solution to this unsatisfactory situation could be found in a {\it reduction} of the VEM space similar to what is done in Finite Elements for quadrilaterals, with the introduction of the Serendipity elements.

\section{Serendipity Virtual Elements in $2$ dimensions}\label{sere2d}

To  fix ideas, and to keep things as simple as possible, we start from  the space $V_{k,k}(\E)$, although, as it will be clearer later on, other choices of the type $V_{k,k_\Delta}(\E)$ are  possible. We recall that if $\E$ has $N_e$ edges, then the dimension of the space will be $N_{\E}:=k\,N_e+\pi_{k,2}$.

\subsection{The property $\Sp$}
Now let us {\it assume} that we have chosen a positive integer $S$ with $\pi_{k,2}\le S\le N_\E$, and  that  the degrees of freedom in \eqref{dofn-0g}-\eqref{dofn-2g} are ordered as
$\delta_1,\,\delta_2,\,...\,\delta_{N_{\E}}$ in such a way that the first $S$ of them, that is{  
\begin{equation}\label{gliS}
\delta_1,\,\delta_2,\,...\,\delta_S
\end{equation}
have the following property: 
\begin{equation}\label{condS}
\!\!\!\!\!(\Sp)\quad\forall p_k\in\P_{k,2}(\E)\quad \{\delta_1(p_k)=\delta_2(p_k)=...=\delta_S(p_k)=0\}\Rightarrow
\{p_k\equiv 0\}.
\end{equation}

As it will become clearer in a while, the $S$ chosen degrees of freedom will be the ones
kept and used in the final system (the other ones being left,
in each element, as ``dummies''). 
 
As a consequence, in order to save the conformity of the whole space (defined on the whole computational domain) it will be always convenient to keep, among the first $S$ degrees of freedom, all the boundary ones \eqref{dofn-0g}-\eqref{dofn-1g}. For simplicity, we will consider only the case
in which this has been done, and we then {\it assume} that: 
\begin{equation}\label{BoundaryS}
\mbox{The d.o.f.s  $~~\delta_1,\,\delta_2,\,...\,\delta_{S}~$
contain all the boundary ones \eqref{dofn-0g}-\eqref{dofn-1g}.}
\end{equation}

In a certain number of cases the boundary degrees of freedom will be sufficient to give the property $\Sp$, but in other cases it will be necessary to add  some {\it internal} degrees of freedom from \eqref{dofn-2g}. The number of these additional  degrees of freedom  will  end up being equal to the number of {\it internal} degrees of freedom
that will be kept in our Serendipity Virtual Elements. Hence it is clear that property $\Sp$ in \eqref{condS} has a crucial relevance, and deserves a more detailed analysis. 

\subsection{Sufficient conditions for property $\Sp$}\label{3.2}

To start with, together with $\eta_{\E}$ it will also be convenient to introduce the
{\it basic bubble} $b_{\E}$ (or simply $b$), that is, the function  given by the product of the equations of the $\eta_{\E}$ different straight lines that contain all the edges of $\E$.

Using 
assumption \eqref{BoundaryS} we  note that a polynomial $p_k\in \P_k$ that satisfies \begin{equation}\label{dofzero}
\delta_1(p_k)=\delta_2(p_k)=...=\delta_S(p_k)=0
\end{equation}  
will be identically zero on all edges of $\partial\E$, and in particular its expression will contain the bubble $b_\E$  as a factor. We also recall that the degree of $b_\E$ is equal to  $\eta_\E$.  Then, in particular, we have
that  a polynomial $p_k$ that satisfies \eqref{dofzero} will necessarily have
the form $p_k=b_{\E} q_{k-\eta_{\E}}$ with $q_{k-\eta_{\E}}$ a polynomial of degree $k-\eta_{\E}$. We will consider, separately, several cases.

\noindent
$\bullet$ {\bf Case $k<\eta_E$}

From the above discussion
we deduce in particular the following result.  
\begin{prop}
For $k<\eta_\E$ assumption \eqref{BoundaryS} implies that property $\Sp$ is always satisfied.
\end{prop}
We then split the analysis of the case $k\ge\eta_\E$ in two cases.

\noindent
$\bullet$ {\bf Case $k\ge\eta_E$ and $\E$ convex}

For values of $k\ge\eta_{\E}$, together with the boundary degrees of freedom, we  would need in  general some additional internal ones. In particular we  have the following result.
\begin{prop}
Assume that $k\ge \eta_\E$,  that $\E$ is convex, and that assumption \eqref{BoundaryS} is satisfied. Assume moreover that  the degrees of freedom $~~\delta_1,\,\delta_2,\,...\,\delta_{S}~$ include  all the moments of order $\le k-\eta_{\E}$ in $\E$ as well.  Then property $\Sp$ is satisfied. 
\end{prop}
\begin{proof} We first note that if $\E$ is {\it convex} then $b_{\E}$ will not change sign inside $\E$.   Hence,  if $p_k$ vanishes on $\partial\E$ (and hence $p_k=b_{{\E}}q^*_{k-\eta_{\E}}$) and if moreover
\begin{equation}\label{ke00}
\int_{\E}p_k\,q\dE =0\quad \forall q\in \P_{k-\eta_{\E}},
\end{equation}
then it is enough to take $q=q^*_{k-\eta_{\E}}$ in \eqref{ke00} to deduce that 
\begin{equation}\label{ke00b}
0=\int_{\E}p_k\,q^*_{k-\eta_{\E}}\dE\, = \, \int_{\E}b_{\E}\,(q^*_{k-\eta_{\E}})^2\dE\quad\mbox{and therefore $p_k=0$}.
\end{equation}
\end{proof}

From the two above propositions we see in particular that: for $k=2$ we will never need internal moments (for any shape of $\E$) and property $\Sp$ will always hold; for $k=3$ we will need the mean value only when $\eta_E=3$, and no internal d.o.f.s for a bigger $\eta_\E$; for $k=4$ we will need all the moments up to the degree $1$ for $\eta_E=3$,  but only the mean value when $\eta_E=4$ {\it and } $\E$ is convex . And so on.

\noindent
$\bullet$ {\bf Case $k\ge\eta_E$ and $\E$ non convex}

The case of non-convex polygons, for $k\ge\eta_{\E}$, is more tricky. For instance if $\E$ is a non convex quadrilateral (as the third case in  Figure \ref{dege}), then $b_\E$ will indeed change sign in $\E$, and the  argument in \eqref{ke00b} will not apply. However, indicating by  $w_2$ the second degree polynomial made by the product of the equations of the two ``re-entrant'' edges, it is easy to check that  the product  $b_\E w_2$  does not change sign inside $\E$ (as the equations of the re-entrant edges will be taken {\it  twice}). The same will obviously be true for more general polygons, whenever we have only two re-entrant edges (as, for instance the  fourth element in Figure \ref{general}). Actually what counts is the number of {\it re-entrant lines}, as in the third example of Figure \ref{general}.  For the sake
of simplicity, however, we restrict ourselves to the case of two re-entrant edges, and present the following result.

\begin{prop}\label{nonconv}
Assume that $k\ge \eta_\E$,  that assumption \eqref{BoundaryS} is satisfied,  and that $\E$ has only two
``re-entrant edges''. Let $w_2$ be the second degree polynomial made by the product of the equations of the two ``re-entrant'' edges.  Assume moreover that  the degrees of freedom $~~\delta_1,\,\delta_2,\,...\,\delta_{S}~$ include also all the moments   
\begin{equation}\label{ke000}
\int_{\E}p_k\,q\,w_2\dE \quad \forall q\in \P_{k-\eta_{\E}}.
\end{equation}
Then property $\Sp$ is satisfied. 
\end{prop} 
\begin{proof} We remark first that if $\E$ has two re-entrant corners then $\eta_\E\ge 4$, and therefore  $k-\eta_\E+2$ (the degree of the test function $q\,w_2$ in \eqref{ke000}) is $\le k-2$, so that the degrees
of freedom in \eqref{ke000} are still part of the degrees of freedom  \eqref{dofn-2} in $V_k(\E)$. Then, let $p_k$
be a polynomial of degree $\le k$ vanishing on $\partial\E$ and such that
\begin{equation}\label{degeq0}
\int_{\E}p_k\,q\,w_2\dE=0 \quad \forall q\in \P_{k-\eta_{\E}}.
\end{equation}
We first deduce, as before, that $p_k=b_{\eta_{\E}}q^*_{k-\eta_{\E}}$ for some $q^*_{k-\eta_{\E}}\in
\P_{k-\eta_{\E}}$. Then we take $q=q^*_{k-\eta_{\E}}$ in \eqref{degeq0} to get
\begin{equation}\label{ke00bb}
0=\int_{\E}p_k\,w_2\,q^*_{k-\eta_{\E}}\dE\, = \, \int_{\E}b_{\E}\,w_2\,(q^*_{k-\eta_{\E}})^2\dE,
\end{equation}
that implies again $p_k=0$ since $b_{\E}\,w_2$ does not change sign in $\E$.
\end{proof}

So far we discussed (long enough) the cases in which assumption $\Sp$ holds true, or it does not. It is now time to see some of its consequences.

\subsection{The operator $\Pi^\mathcal{S}_k$}

As we shall see in a little while, given a set of degrees of freedom $\delta_1,\,\delta_2,\,...\,\delta_S$ (subset of
 \eqref{dofn-0g}-\eqref{dofn-2g}) that satisfy property $\Sp$ (see \eqref{condS}), it will always be possible to construct an operator $\Pi^\mathcal{S}_k$ from  $V_{k,k}(\E)$ to $\P_k(\E)$ with the following properties:

\begin{equation}\label{propR1}
\bullet \mbox{ $\Pi^\mathcal{S}_k$ is computable using only the d.o.f.  $\delta_1,...,\delta_{S}$} ,
\end{equation}and
\begin{equation}\label{propR2}
\bullet \mbox{ $\Pi^\mathcal{S}_k  q_k = q_k$ for all $q_k\,\in\P_{k}$} .
\end{equation}

\subsection{The reduced (Serendipity) VEM spaces}

Once the operator $\Pi^\mathcal{S}_k$ has been defined, we can use it to construct our Serendipity VEM spaces. The basic idea can be summarized as follows.
\begin{itemize} 
\item we work in $V_{k,k}(\E)$,
\item for each $\vphi\in V_{k,k}(\E)$ we use the first $S$ degrees of freedom
to construct $\Pi^\mathcal{S}_k \vphi$, 
\item then we use $\delta_r(\Pi^\mathcal{S}_k\vphi)$, for $S<r\le N_\E$ to define the values of the remaining $N_{\E}-S$ degrees of freedom  in $V_{k,k}(\E)$ . 
\end{itemize}

In other words, given $\vphi\in V_{k,k}(\E)$ we construct another element (say, $\widetilde{\vphi}$)
such that 
\begin{equation}\label{S1}
\delta_r(\widetilde{\vphi})=\delta_r(\vphi) \,\mbox{ for }\,(1\le r\le S),
\end{equation}
and
\begin{equation}\label{S2}
 \delta_r(\widetilde{\vphi})=\delta_r(\Pi^\mathcal{S}_k\vphi) \,\mbox{ for }\,(S+1\le r\le N_\E) .
\end{equation}
Clearly,  the elements $\vphi\in V_{k,k}(\E)$ such that $\widetilde{\vphi}=\vphi$ form the space 
\begin{equation}\label{VEMS2d}
V^S_k(\E)=\{\vphi\in V_{k,k}(\E)\mbox{ s. t. }\delta_r(\vphi)=\delta_r(\Pi^\mathcal{S}_k\vphi)\;\forall r = S+1,...,N_{\E}\},
\end{equation}
that we identify as our {\it reduced (Serendipity) Virtual Element Space}. It is immediate to see that the space
$V^S_k(\E)$ has the following properties:
\begin{itemize}
\item the dimension of $V^S_k(\E)$ is $S$, 
\item $\delta_1,...,\delta_{S}$ is a unisolvent set of degrees
of freedom for $V^S_k(\E)$, 
\item  $\P_{k,2}(\E)\subseteq V^S_k(\E)$,
\item the $L^2$-projection $\Pi_k^0$ is computable from the d.o.f. of  $V^S_k(\E)$.
\end{itemize}
It is also immediate to see that  {\bf for triangles}  the new spaces
$V^S_k(\E)$ have now the same number of degrees of freedom as the classical Lagrange Finite Elements, and are,
actually, the same spaces, since $\P_{k,2}(\E)$ and $ V^S_k(\E)$ have the same dimension. See Figure \ref{seretri}. 
\begin{figure}[!h]
  \begin{center}
    \includegraphics[width=7.7cm]{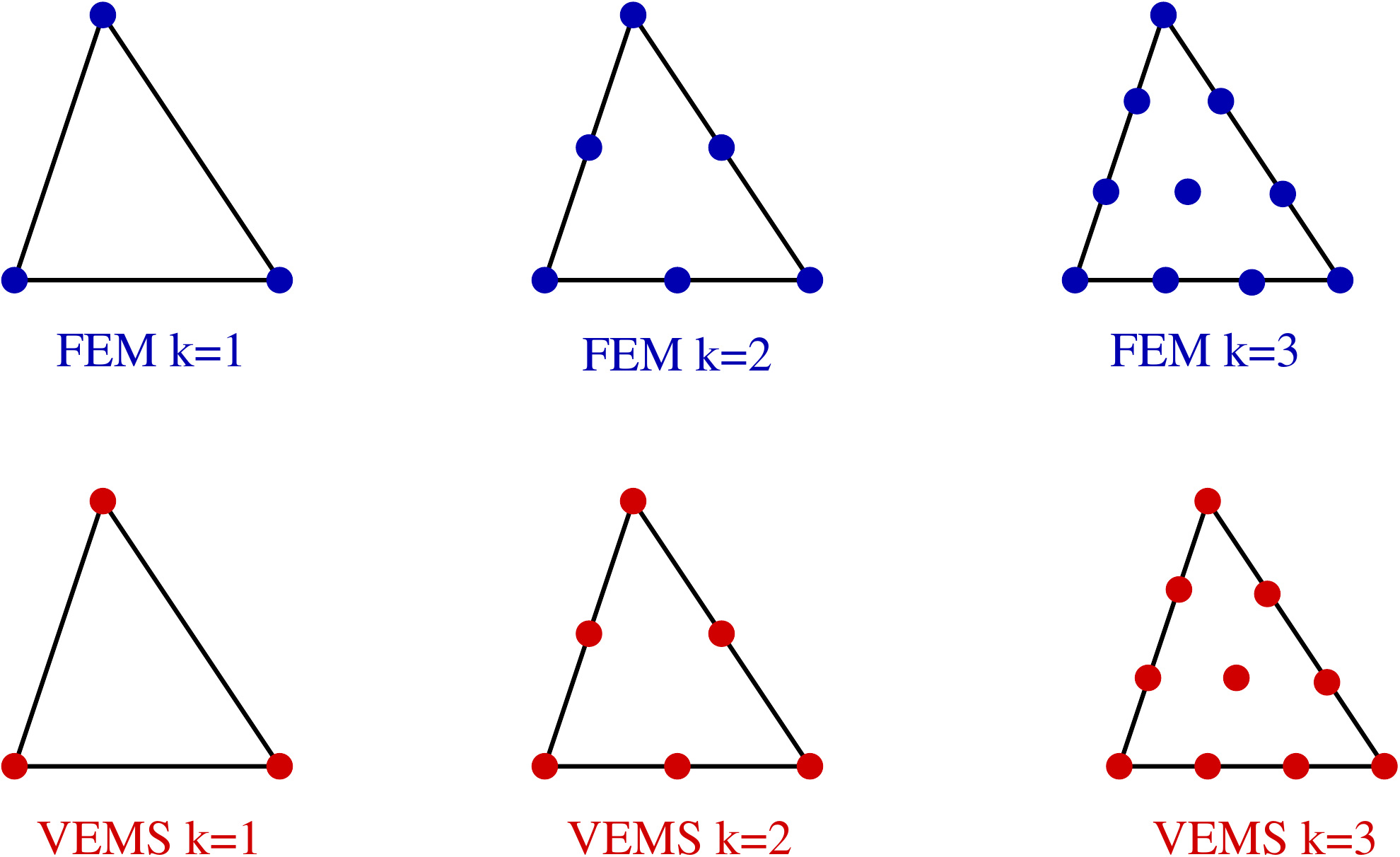}
    \end{center}
   {\caption{Triangles: Classical FEM and Serendipity VEM}\label{seretri}}
\end{figure}

 \bigskip 

Serendipity Finite Elements on quadrilaterals are in general defined on squares and on their affine images
(that is, on parallelograms),
while their extension to more general quadrilaterals (via isoparametric mappings) suffers, in general, a loss
of accuracy (see e.g. \cite {A-B-F-nod}).

{\bf For parallelograms}, our Serendipity Virtual Elements have  the same number of degrees of freedom as the Serendipity Finite Elements: for a general $k$ both  use the boundary degrees of freedom plus the internal moments of degree $\le k-4$, although, in general, with a different space. 

{\bf For more general quadrilaterals} Serendipity Virtual Elements and Serendipity Finite Elements have again
the same number of degrees of freedom (see Figure \ref{serequa}, and, for instance, papers
 \cite{A-A-Sere} or \cite{F-G-Sere}), although Finite Elements allow much less general distortions, and even for small deviations from parallelograms show a lack of accuracy that disappears only if the mesh  (progressively, as the mesh-size $h$ goes to zero) tends to be made of parallelograms
(see \cite{A-B-F-nod}).
  \begin{figure}[!h]
  \begin{center}
    \includegraphics[width=7.7cm]{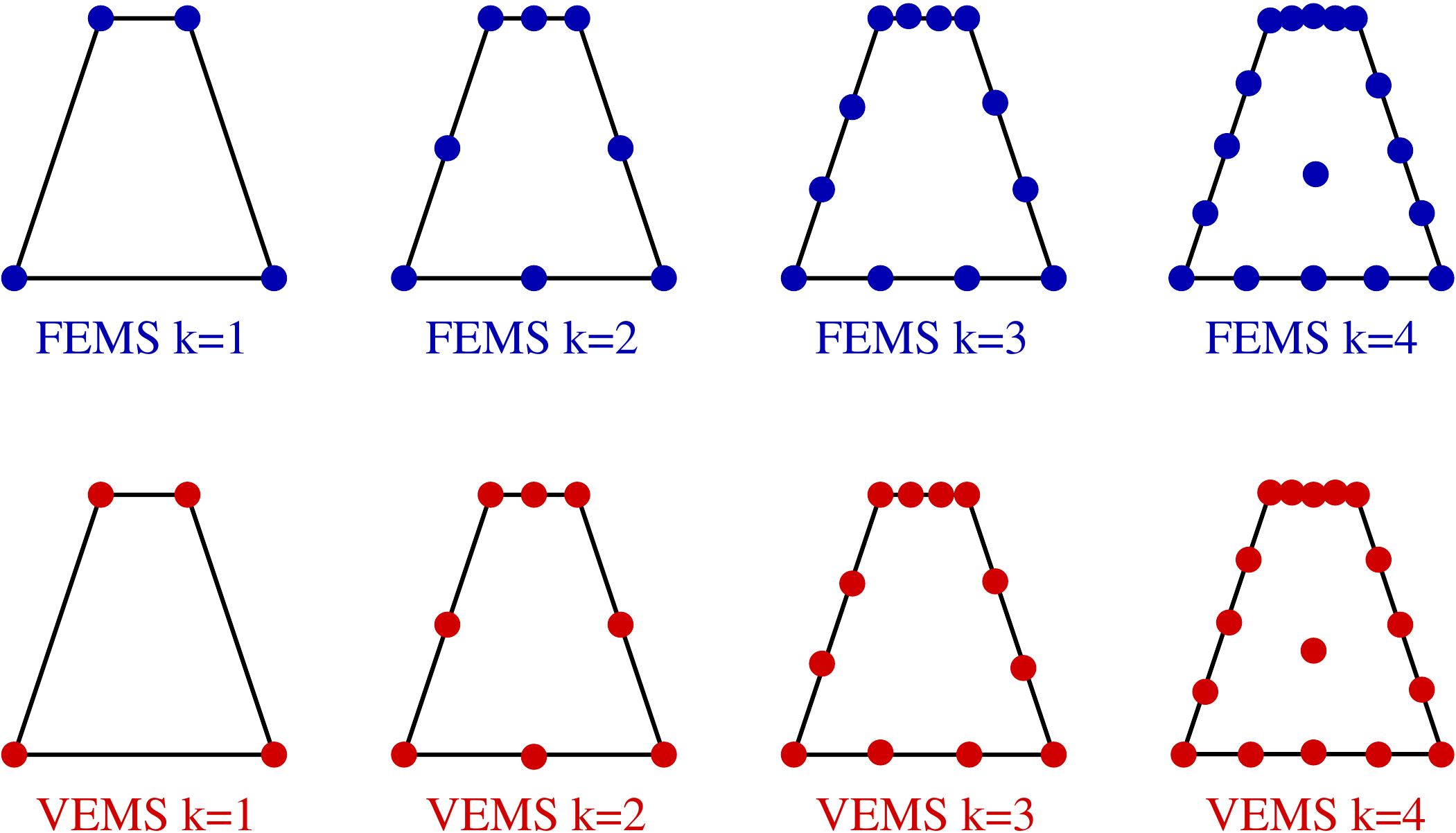}
  \end{center}
  {\caption{Quads: S-FEM (Arnold-Awanou) and S-VEM}\label{serequa}}
\end{figure}
On the other hand, Virtual Elements are extremely robust, and can survive several types of severe
distortion. The only degeneration that  must be avoided, in the present context,  occurs, clearly, when two edges fit in the same straight line (as,  for instance, in the second example of  Figure \ref{dege}).  But even when the element degenerates to a triangle  we could still  survive in a cheap-and-easy way, just by using also the internal moments of degree  up to $k-3$. Clearly, for stability reasons, when two edges are {\it almost on the same straight line} it would still be wise to use also the moments of degree $k-3$. Hence we can say that for  quadrilateral elements we have  the same number of degrees of freedom that Serendipity Finite Elements use on affine elements, but our construction works in much more general cases, using a different space that is more robust to distortions. 
 In Figure \ref{dege} we show some example of allowed distortions. In the first case depicted, only moments of degree up to $k-4$ need to be included, while in the second case also the moments of degree $k-3$ are needed. In the third case we can use moments of degree up to $k-4$ with the quadratic multiplicative factor defined in \eqref{ke000}.

\begin{figure}[htbp]
  \begin{center}
    \includegraphics[width=7.7cm]{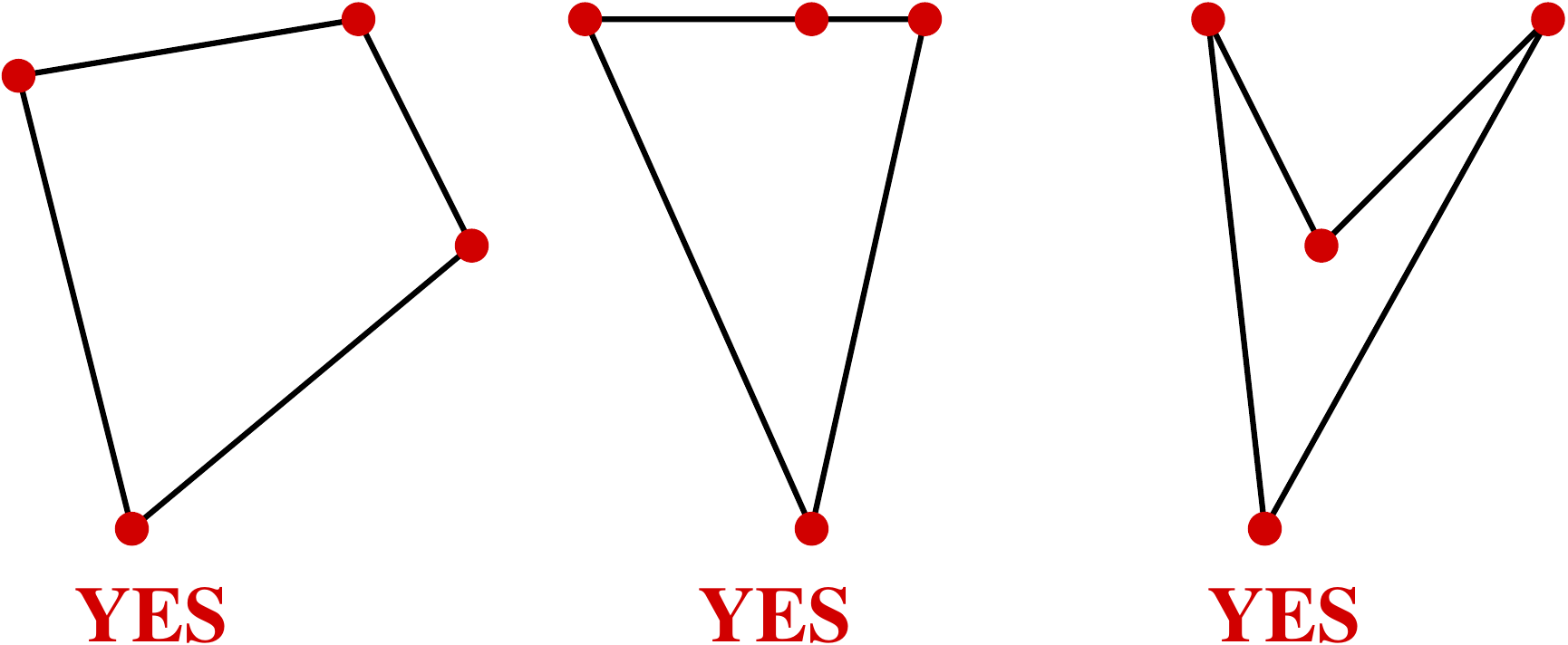}
  \end{center}
  {\caption{Allowed distortions for quadrilaterals}\label{dege}}
\end{figure}

Finally it is still worth mentioning that Serendipity VEM can also be defined (and perform very well) on {\bf much more general
polygons} where Serendipity Finite Elements (as well as classical Finite Elements) do not exist.

\subsection{Construction of $\Pi^\mathcal{S}_k$}\label{costR}

There is just one item that we have to detail in order to complete the description of the nodal Serendipity Virtual Elements on polygons: the construction of the operator $\Pi^\mathcal{S}_k$ starting from a set of degrees of freedom that satisfy property $\Sp$. For this, we assume that, for a given $k$, we are given  a set $\delta_1,\,\delta_2,\,...\,\delta_S$ of degrees of freedom having the property $\Sp$, and we define the operator $\calD$
\begin{equation}
\calD: V_{k,k}(\E)\rightarrow\,\R^S\;\mbox{ defined by }\;
\calD\vphi:=(\delta_1(\vphi),...,\delta_S(\vphi)).
\end{equation}
Needless to say, the operator $\calD$ will have the properties: 
\begin{equation}\label{propD1}
\bullet \mbox{ $\calD$ can be computed using only the d.o.f  $\delta_1,...,\delta_{S}$},
\end{equation}
\begin{equation}\label{propD2}
\hskip-3.1truecm\bullet\,\,\calD\, q=0\Rightarrow q=0\quad\mbox{ for all }q\in\P_{k}.
\end{equation}
Property \eqref{propD1} is trivial, and property \eqref{propD2} is  inherited by \eqref{condS}.

We observe that, for coding purposes, the operator $\calD$ corresponds to take the first $S$ rows of the matrix $\bf D$ given in \cite{hitchhikers}, formula (3.17).

We are now going to use $\calD$ to construct $\Pi^\mathcal{S}_k$ as follows: for every $\varphi\in V_{k,k}(\E)$   we can define $\Pi^\mathcal{S}_k\,\varphi\in\P_k$ through
\begin{equation}\label{defR}
(\calD(\Pi^\mathcal{S}_k\varphi-\varphi),\calD q)_{\R^{S}}=0\quad\forall q\in \P_k,
\end{equation}
where $(\cdot\, ,\,\cdot)_{\R^{S}}$ is the {\it Euclidean scalar product} in $\R^{S}$
(or, if more convenient, any positive definite symmetric bilinear form on $\R^{S}$). Property \eqref{propD2} ensures that the matrix
\begin{equation}
(\calD p,\calD q)_{\R^{S}} \qquad p,\,q\in \P_k
\end{equation}
is nonsingular, so that for every right-hand side $(\calD \varphi,\calD q)_{\R^{S}}$
the linear system \eqref{defR} in the unknown $\Pi^\mathcal{S}_k\,\varphi$ will have a unique solution. 
It is an easy exercise to check that the operator $\Pi^\mathcal{S}_k $, as defined in \eqref{defR}, satisfies the required properties \eqref{propR1}-\eqref{propR2}. 

\subsection{Different options}

 We first point out that, in our presentation, the
reason why we delayed the construction of the operator $\Pi^\mathcal{S}_k$ is the presence, in
its construction, of an excessive freedom. Indeed, there are zillions of possible
choices for the basic degrees of freedom to be used (in the construction of the
operator $\calD$) and zillions of possible choices for the symmetric and positive
definite bilinear form to be used (if convenient) in place of the
{\it Euclidean scalar product} $(\cdot\, ,\,\cdot)_{\R^{S}}$ in $\R^{S}$. In
principle, the presence of many choices could allow a strategy toward a final space
with suitable properties (we shall see an example later on). But in many cases the
presence of too many options is more a drawback than an advantage.

We did not consider so far the
{\it scaling} and {\it stability} problems. As pointed out in several occasions
(actually, almost everywhere) in the VEM literature, it is (much) wiser to use
degrees of freedom that {\it scale} in the same way. Otherwise (for instance) the
choice of the Euclidean scalar product should not be recommended, since  degrees of
freedom that scale differently should be treated in different ways.

It should be said, however, that the situation is not as bad as it could seem.
Indeed, once we took care of choosing degrees of freedom that scale in the same way,
the methods show a remarkable robustness, and the use of
the Euclidean scalar product, or of the Euclidean scalar product multiplied or
divided by 10, or of other similar bilinear forms, would end up in equally good
final schemes.

\subsection{The lazy choice and the stingy choice}\label{stingy-lazy}
 We have seen that, for an order of accuracy $k$, and
for a polygon { (for simplicity, convex)} whose edges belong to $\eta_E$ different straight lines, 
in our serendipity spaces only internal moments up to the degree $k-\eta_E$ can be used.
We also pointed out that, however, for stability reasons one should also  take care
of the cases where two (or more) edges belong {\it almost} to the same straight
line, and consider them {\it as actually belonging to the same straight line}. This
would decrease  the number $\eta_E$ for the polygon, and increase the number $k-\eta_E$
of moments to be used. An additional difficulty, with this choice, would then be to
decide the precise meaning of the above
term ``{\it almost}'', for instance in terms of the angle between the two ({\it
almost} coincident) straight lines that contain the  two (or more) edges under
scrutiny.


 In light of the above discussion  (and always  for a
given fixed order of accuracy $k$) we see that, in the actual implementation  of a
code in which many different shapes of polygonal elements are expected, one faces a very important choice. {\it A first possibility} (let us call it, {\bf the
stingy choice}) would be: to fix a minimum angle $\theta_0>0$ and then, for every
polygon $\E$, to count the number $\eta_E(\theta_0)$ 
of different straight lines
that
contain all the edges of $\E$, by considering ``different from each other'' two
straight lines only when the smaller angle between them is bigger than $\theta_0$. Then,
use moments up to the order $k-\eta_E(\theta_0)$ as degrees of freedom inside $\E$.
{\it Another  possibility} (let us call it, {\bf the lazy choice})  would be to  use always internal moments of
degree up to $k-3$, since our assumptions imply that $\eta_E(\theta_0)$ is always $\ge 3$ for $\theta_0$ small enough compared to $\rho_0$ (say, for $\rho_0\ge\tan(\theta_0/2) $).
Needless to say, many strategies in between are
possible, and the choice among all of them would depend on the type of code one is
writing, and on the use one wants to make of it. We shall come back to this problem
when dealing with the three-dimensional case.

\section{Serendipity Virtual Elements in $3$ dimensions}\label{sere3d}


Let us consider now the case of three-dimensional VEM.  Again, for the sake of simplicity, we will make some simple assumptions on the geometry of our elements. In particular we will consider the typical assumption (see for instance \cite{projectors}): there exists a fixed number $\rho_0>0$, independent of the decomposition, such that for every polyhedron $\PP$ (with diameter $h_\PP$) we have that: {\it i)} $\PP$ is star-shaped with respect of all the points of a ball of radius $\rho_0\,h_\PP$, {\it ii)} every edge $e$ of $\PP$ has length $|e|\ge \rho_0\,h_\PP$, and {\it iii)} every face $f$ is star-shaped with respect of all the points of a ball of radius $\rho_0\,h_\PP$. Here too, more general assumptions could be allowed  but again this goes beyond the scope of the present paper.  See for instance \cite{projectors}. 

As we did for the two-dimensional case, we shall concentrate on the choice of the spaces on a single polyhedron $\PP$.

Moreover, still to keep things as simple as possible, we assume that, in the terminology
of Subsection \ref{stingy-lazy} , we follow for every {\it face}  the {\bf lazy} choice.

\subsection{Polynomials that vanish on $\partial\Pt$}

We point out that, for the faces of a three-dimensional decomposition, the difference between the two choices (stingy and lazy) would be decidedly more dramatic than in two dimensions. Indeed, 
for 2D-decompositions the degrees of freedom internal to the elements could always be eliminated (easily and  cheaply) by {\bf static condensation}. But in three dimensions the degrees of freedom internal to faces cannot be (easily and cheaply) eliminated by static condensation, and in general they still appear in the final (global) stiffness matrix. The difference would become more and more expensive for higher choices of the accuracy $k$. To make an example, for $k=8$ on an hexagonal face $f$ (with $\eta_f=6$) the lazy choice would require the use of all the moments of degree up 
to $8-3$ (that is, $21$ d.o.f.) while the stingy choice would require only the moments of degree up to $8-6$ (that is, $6$ d.o.f.). Hence, the systematic use 
of the lazy choice on all faces (as done here) is more a way of keeping the presentation simple rather than a suggestion on what to do in a practical code. Indeed, for higher order of accuracy and for decompositions in which many faces have (each) many edges, we would {\bf not} recommend the lazy choice, which could be much more expensive.
We think, however, that once the basic idea is understood it will be quite immediate for the users to see how and when to shift from the lazy choice to more cheap ones.

We then take  an integer $k\ge 1$ and we consider for every face $f$  { (that for simplicity we assume to be convex)} the 
Serendipity space $V_k^S(f)$ (as we said, to fix ideas, with the lazy choice). 

Then for $\kD\ge-1$ we define the space 
\begin{multline}\label{space3d}
V_{k,\kD}(\PP):=\{\vphi\in C^0(\overline{\PP})\mbox{ such that }\\ \vphi_{|f} \in V_k^S(f)
\;\forall\mbox{ face }f\mbox{ in }\partial\PP,
\mbox{ and }
\Delta\vphi\in\P_{\kD}(\PP)
\}
\end{multline}
with the degrees of freedom
\begin{equation}\label{dofn-03d}
\bullet\;\,\mbox{the values of $\varphi$ at the vertices} ,
\end{equation}
\begin{equation}\label{dofn-13d}
\hskip1.2truecm\bullet \int_e\varphi\, q\ds\quad \forall \mbox{ edge }e 
\mbox{ for all }q\in\P_{k-2}(e) ,
\end{equation}
\begin{equation}\label{dofn-23d}
\hskip1.2truecm\bullet \int_f\varphi\, q\df\quad \forall \mbox{ face }f \mbox{ for all }q\in\P_{k-3}(f) .
\end{equation}
\begin{equation}\label{dofn-33d}
\bullet \int_\PP \varphi\, q\dP\quad \mbox{ for all }q\in\P_{\kD}(\PP) .
\end{equation}
We point out that  the degrees of freedom \eqref{dofn-23d} follow from our decision  to  always take the lazy choice on every face { and from the simplified assumption of convex faces}.
For non convex faces we should adapt the nature of the degrees of freedom (although, in general, not the number), as discussed in Subsection \ref{3.2}. 

\subsection{$\calD$, $\Pi^\mathcal{S}_k$,  and the Serendipity spaces}

At this point we could restart {\it mutatis mutandis} the reduction procedure that we
followed for the two-dimensional case. The two cases (2-dimensional and 3-dimensional)
are very similar, and therefore we will summarize the 3-dimensional one very shortly.

We start by taking $\kD=k$ in \eqref{space3d} as we did at the beginning of Section \ref{sere2d}. Let $N_{\PP}$ be the number of degrees of freedom of  $V_{k,k}(\PP)$.
We order them in such a way that
the boundary ones  \eqref{dofn-03d}-\eqref{dofn-23d} come first (and, typically, the internal moments are ordered  
from lowest to highest degree). Then we choose an integer $S$ such that the first $S$
degrees of freedom are: the boundary ones, and the internal moments of degree up
to $k-\eta_{\PP}$, where now, in general, $\eta_{\PP}$ is the number of {\it distinct} planes that contain all the faces of $\PP$. Here too, we could make the lazy choice of taking always $\eta_{\PP}=4$.

We note that our degrees of freedom will satisfy the property (that we still call $\Sp$):
\begin{equation}\label{condS3}
\!\!\!\!(\Sp)\;\;\forall p_k\in\P_{k,3}(\PP),\quad \{\delta_1(p_k)=\delta_2(p_k)=...=\delta_S(p_k)=0\}\Rightarrow
\{p_k\equiv 0\} ,
\end{equation}
and therefore we can use them to construct, following the same path that we took in Subsection \ref{costR}, a projection operator $\Pi^\mathcal{S}_k$ such that :
\begin{equation}\label{propR13}
\bullet \mbox{ $\Pi^\mathcal{S}_k$ is computable using only the d.o.f.  $\delta_1,...,\delta_{S}$} ,
\end{equation}and
\begin{equation}\label{propR23}
\bullet \mbox{ $\Pi^\mathcal{S}_k q_k = q_k$ for all $q_k\,\in\P_{k}$}. 
\end{equation}
 Once we have the operator $\Pi^\mathcal{S}_k$ we can define the Serendipity Virtual Element space $V_k^S(\PP)$
as
\begin{equation}\label{VEMS2dkk}
V^S_k(\PP)=\{\vphi\in V_{k,k}(\PP)\mbox{ s. t. }\delta_r(\vphi)=\delta_r(\Pi^\mathcal{S}_k\vphi)\;\forall r = S+1,...,N_{\PP}\}.
\end{equation}
As degrees of freedom for the space $V^S_k(\PP)$,  defined in \eqref{VEMS2dkk}, we take
\begin{equation}\label{dofn-03dS}
\bullet\;\,\mbox{the values of $\varphi$ at the vertices} ,
\end{equation}
\begin{equation}\label{dofn-13dS}
\hskip1.2truecm\bullet \int_e\varphi\, q\ds\quad \forall \mbox{ edge }e 
\mbox{ for all }q\in\P_{k-2}(e) ,
\end{equation}
\begin{equation}\label{dofn-23dS}
\hskip1.2truecm\bullet \int_f\varphi\, q\df\quad \forall \mbox{ face }f \mbox{ for all }q\in\P_{k-3}(f) ,
\end{equation}
\begin{equation}\label{dofn-33dS}
\bullet \int_\PP \varphi\, q\dP\quad \mbox{ for all }q\in\P_{k-\eta_{\PP}}(\PP) ,
\end{equation}
and we point out that in \eqref{dofn-23dS} we could use, for each face $f$, the moments only up to the degree $k-\eta_{f}$ if we chose
a more stingy strategy. Just to make a toy-example, on a regular dodecahedron ($12$ pentagonal faces, with a total of $20$ vertexes and $30$ edges) for $k=4$ we would have, with the most stingy choice (on faces and inside), only one d.o.f. per vertex and three additional 
degrees of freedom per edge (for a total of $110$ degrees
of freedom: the absolute minimum, if you want a $\P_4$ conforming element). The original VEMs would have required $12\times\pi_{2,2}+1\times\pi_{2,3}=82$ additional degrees of freedom
($6$ for each of  the $12$ faces , and $10$ for the interior of the polyhedron). Adopting the lazy choice, instead, we would add (to the $110$ ones on vertices
and edges) $3$ degrees of freedom per face and one inside (for a total of 37 additional d.o.f.s).

\begin{remark} \label{misti}
The extension of the present idea to construct  a Serendipity version of $H({\rm div})$ and $H({\bf curl})$-conforming vector valued spaces 
(as the ones in \cite{super-misti}) can be done in a reasonably easy way,  and is the object of a paper in preparation (by the same authors).
\end{remark}

%
\subsection{Different degrees of freedom}
%

An obvious generalization of our procedure (among several others) would be (for simplicity: in two dimensions) to substitute part of  the original degrees of freedom \eqref{dofn-0}-\eqref{dofn-2} with some equivalent ones. For instance, for $k\ge 2$
one can use, instead of the moments \eqref{dofn-1},  the values of $\vphi$ at $k-1$ nodes inside each edge (a typical convenient choice would be given by the $k-1$ Gauss-Lobatto nodes inside the edge).

     Another example has been suggested already in Proposition \ref{nonconv}: for non convex polygons, we could use suitable
    polynomial weights in the degrees of freedom, including the equations (among those defining the edges) that change sign
    inside $\E$ .

But more imaginative variants could come out being convenient
in some circumstances.  In particular, it is not necessary that the functionals in \eqref{gliS} (the ones used to construct $\calD$ and then $\Pi^\mathcal{S}_k$), are a subset of the original degrees of freedom: we only need to select $S$
linear functionals, and then, if convenient, use in \eqref{gliS} a different set of d.o.f.s that can be deduced from the  chosen ones.

 For instance, one could keep  the nodal values \eqref{dofn-0} and the moments \eqref{dofn-1} as degrees of freedom  (for obvious conformity reasons), but then use in \eqref{gliS}, in place of \eqref{dofn-0} and  \eqref{dofn-1}: 
 \begin{equation}\label{meanv}
 \bullet\mbox{ the mean value of $\varphi$ over $\partial\E$}
 \end{equation}
  and (after ordering the vertices $V_1,...,V_N, V_{N+1}\equiv V_1$ in the, say, counterclockwise order) the integrals 
  \begin{equation}\label{weird}
  \bullet\;
I_{j,k}:=\int_{V_j}^{V_{j+1}}\frac{\partial\varphi}{\partial t}q_{k-1}\ds \mbox{ for } j=1,2,...,N
\mbox{ and }q_{k-1}\in\P_{k-1}
\end{equation}
(under the obvious condition that  $\sum_j I_{j,1}\equiv \varphi(V_{N+1})-\varphi(V_1)=0$). Clearly, as we said, the boundary degrees of freedom would remain \eqref{dofn-0}-\eqref{dofn-1}, but the new ones (that is, \eqref{meanv} and \eqref{weird}) could be employed (possibly together with other data) to define $\calD$ and then to construct $\Pi^\mathcal{S}_k$.  A choice like this might be interesting when combining Serendipity VEM spaces of various nature (like, say, the nodal ones here and the edge-ones mentioned in Remark \ref{misti}  above).

\section{Numerical experiments}

As pointed out before, the Serendipity variant of the Virtual Element Method raises
several problems of computational nature, 
like for instance the definition of $\eta_E$ in the case
of almost-degenerate polygons, or the choice of the scalar product in the definition
of the projector ${\Pi^\mathcal{S}_k}$.

In this paper we will limit ourselves to the presentation of very simple numerical
experiments showing that the method works as expected for an elliptic equation 
in two cases: quadrilateral
elements and a more general Voronoi mesh made of convex polygons. In both cases
we have taken $k=2,3,4$. The error shown is always the relative $L^2$ error;
the $H^1$ error behaves similarly.

\begin{figure}[ht]
\hfill
 \begin{minipage}[b]{0.49\textwidth}
  \begin{center}
  \includegraphics[width=0.9\textwidth]{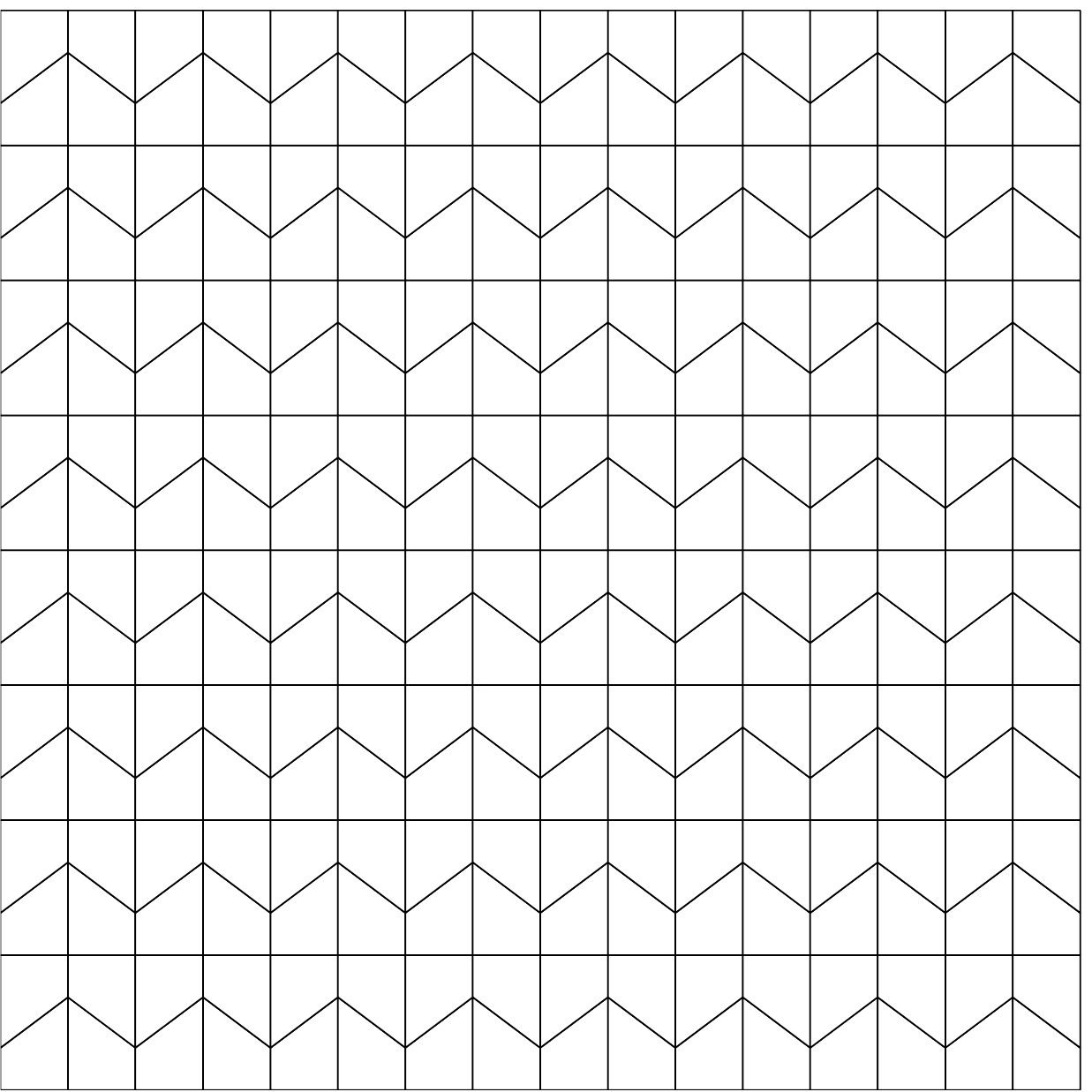}
  \end{center}
  \caption{Trapezoidal mesh}
 \label{fig:boffi}
 \end{minipage}
\hfill
 \begin{minipage}[b]{0.49\textwidth}
  \begin{center}
  \includegraphics[width=0.9\textwidth]{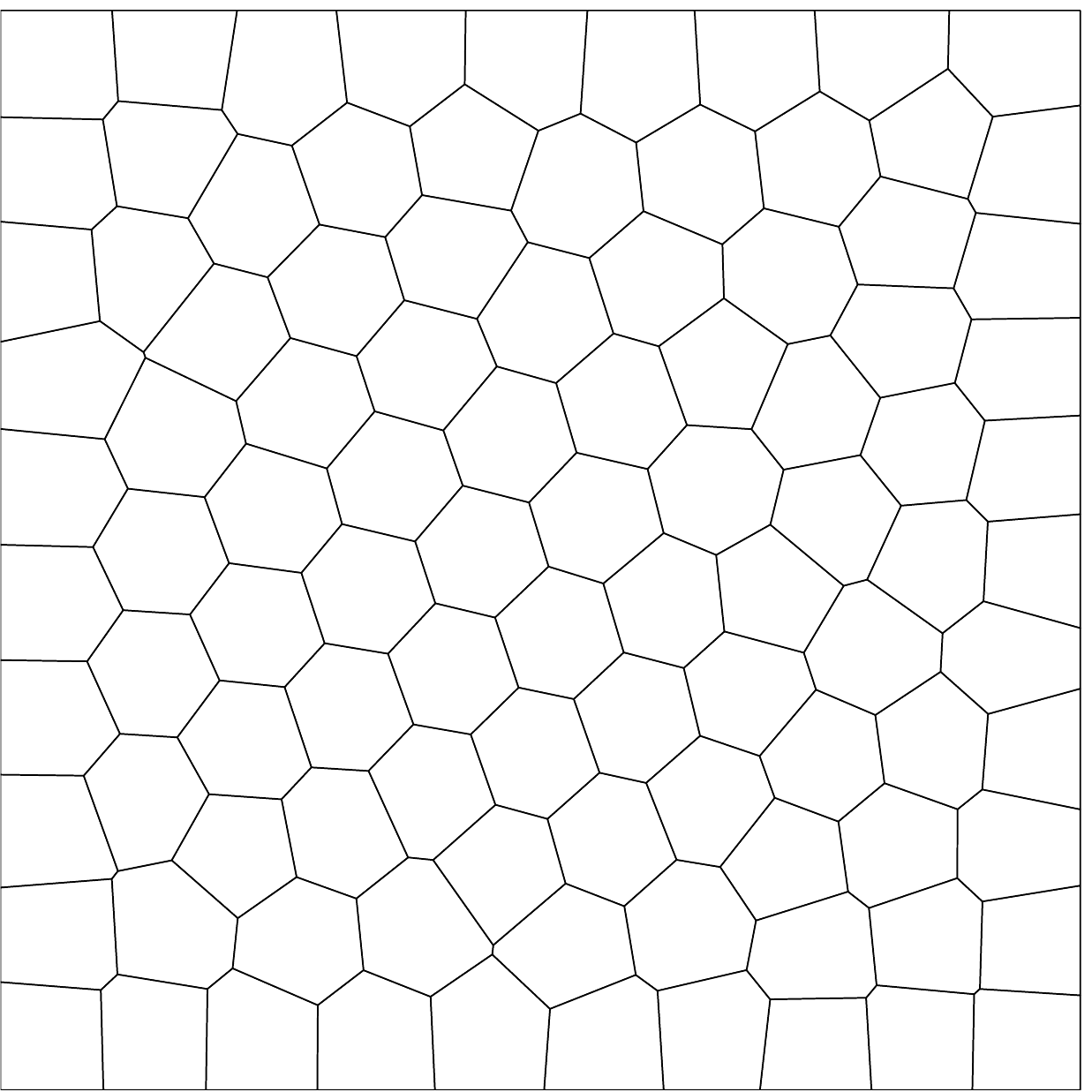}
  \end{center}
  \caption{Voronoi mesh}
 \label{fig:lloyd}
 \end{minipage}
\hfill
\end{figure}
We set $\Omega=]0,1[^2$ and consider the elliptic problem
\begin{equation}\label{Pb-cont}
\left\{
\begin{aligned}
\div(-\diffp\nabla\p + \bb\p) + \reaction\,\p
&= f\quad\text{in }\Omega\\
\p &= g\quad\text{on }\partial\Omega.\\
\end{aligned}
\right.
\end{equation}
The variational form of problem \eqref{Pb-cont} is given by
\begin{equation}\label{tre-uno}
\int_{\Omega}\diffp \nabla \p\cdot \nabla \q\,{\rm d}x
-\int_{\Omega} \p (\bb\cdot\nabla \q) \,{\rm d}x
+\int_{\Omega}\reaction \p \,\q\,{\rm d}x
=
\int_{\Omega}f\,\q
\end{equation}
and, as shown in \cite{variable-primal}, its Virtual Element approximation
consists in replacing in each element
\begin{equation}
\p\ \ \text{with}\ \ \Pi^0_{k-1}\p_h
\qquad\text{and}\qquad
\nabla\p\ \ \text{with}\ \ \Pi^0_{k-1}\nabla\p_h.
\end{equation}
The difference with respect to \cite{variable-primal} is that here the 
$L^2$ projections are computed using the operator $\Pi_{k}^\mathcal{S}$ 
instead of $\Pi^\nabla_k$
for the missing moments.
The stabilization term is defined in terms of the $L^2$-projection.

\subsection{Quadrilateral meshes}

In the quadrilateral case we have considered the trapezoidal mesh studied 
in \cite{A-B-F-nod} for which the authors have proved that the classical serendipity
finite elements do not converge with the optimal rates. We have compared our
serendipity VEM with the classical serendipity finite elements $\aleS_k$
and with the standard $\aleQ_k$ elements. The sequence is composed of four meshes
with $8\times8$, $16\times16$, $32\times32$ and $64\times64$ trapezoids respectively.
In Fig. \ref{fig:boffi} the $16\times16$ mesh is shown.

We have considered the Poisson problem, i.e. we have taken in \eqref{Pb-cont}
\begin{equation}
\diffp = \begin{pmatrix}1&0\\0&1\end{pmatrix},\quad
\bb=(0,0),\quad
\reaction = 0,
\end{equation}
with the right hand side $f$ and the Dirichlet data $g$ defined in such a way that
the exact solution is the fifth-degree polynomial
\begin{equation}
\pex(x,y) := x^3+5y^2-10y^3+y^4+x^5+x^4y.
\end{equation}
In Figs \ref{fig:boffi-L2-k=2}, \ref{fig:boffi-L2-k=3} and \ref{fig:boffi-L2-k=4}
we show the relative $L^2$ error for the three methods. We observe that the
serendipity VEM (``stingy'') behaves like the $\aleQ_k$ element but with much fewer degrees
of freedom.

\begin{figure}[ht]
\hfill
 \begin{minipage}[b]{0.49\textwidth}
  \begin{center}
  \includegraphics[width=\textwidth,height=6.3cm]{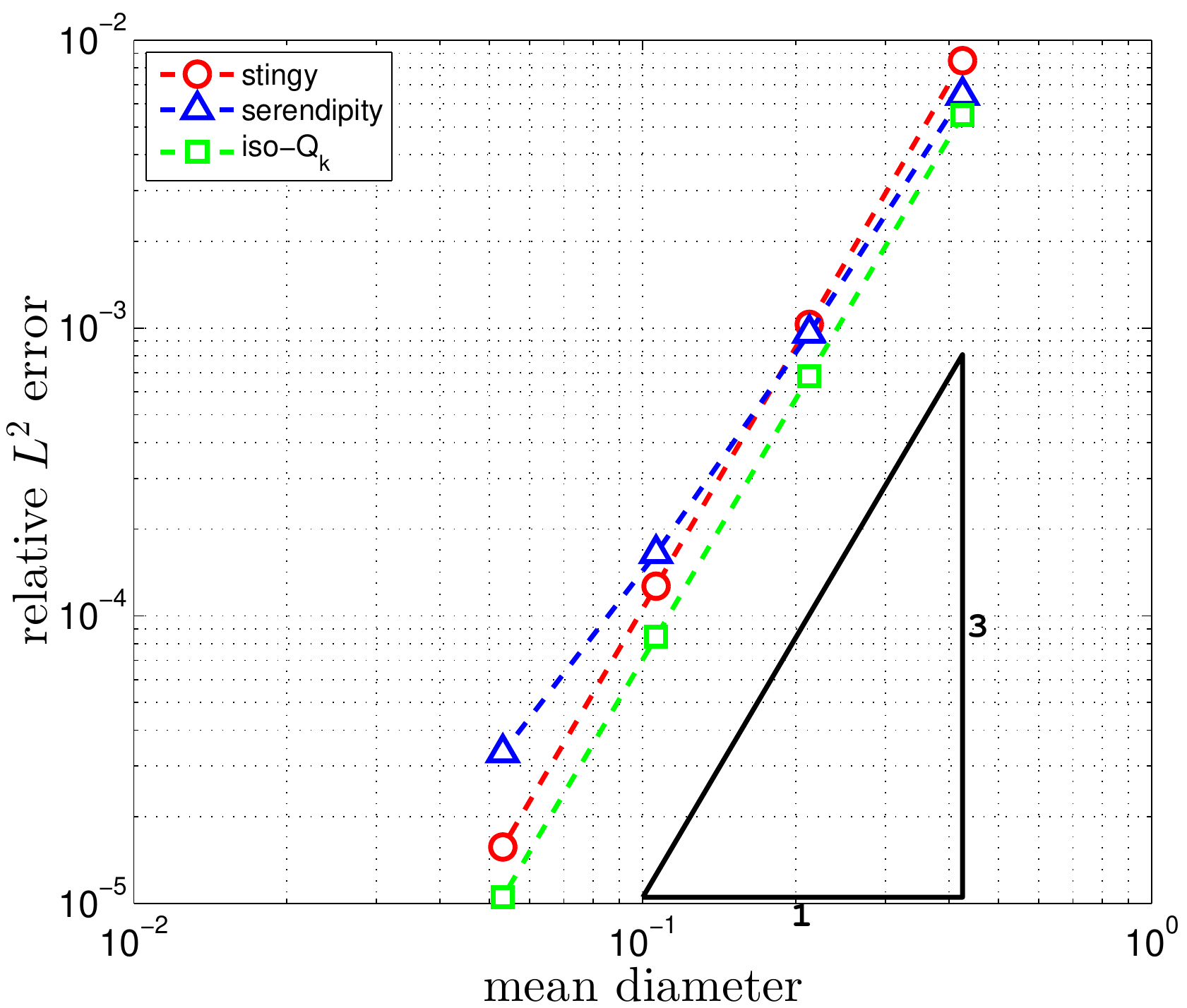}
  \end{center}
 \end{minipage}
\hfill
 \begin{minipage}[b]{0.49\textwidth}
  \begin{center}
	\include{./tableBBMR/table-boffi-k=2-QkSk}

	\bigskip
	\bigskip
	\end{center}
 \end{minipage}
  \caption{$k=2$, $L^2$ error for the trapezoidal meshes.
   Note the non-optimal convergence
   rate for the classical serendipity finite element method $\aleS_k$
   compared with the serendipity VEM (``stingy''); both have the same number
   of degrees of freedom.}
 \label{fig:boffi-L2-k=2}
\hfill
\end{figure}

\begin{figure}[ht]
\hfill
 \begin{minipage}[b]{0.49\textwidth}
  \begin{center}
  \includegraphics[width=\textwidth,height=6.3cm]{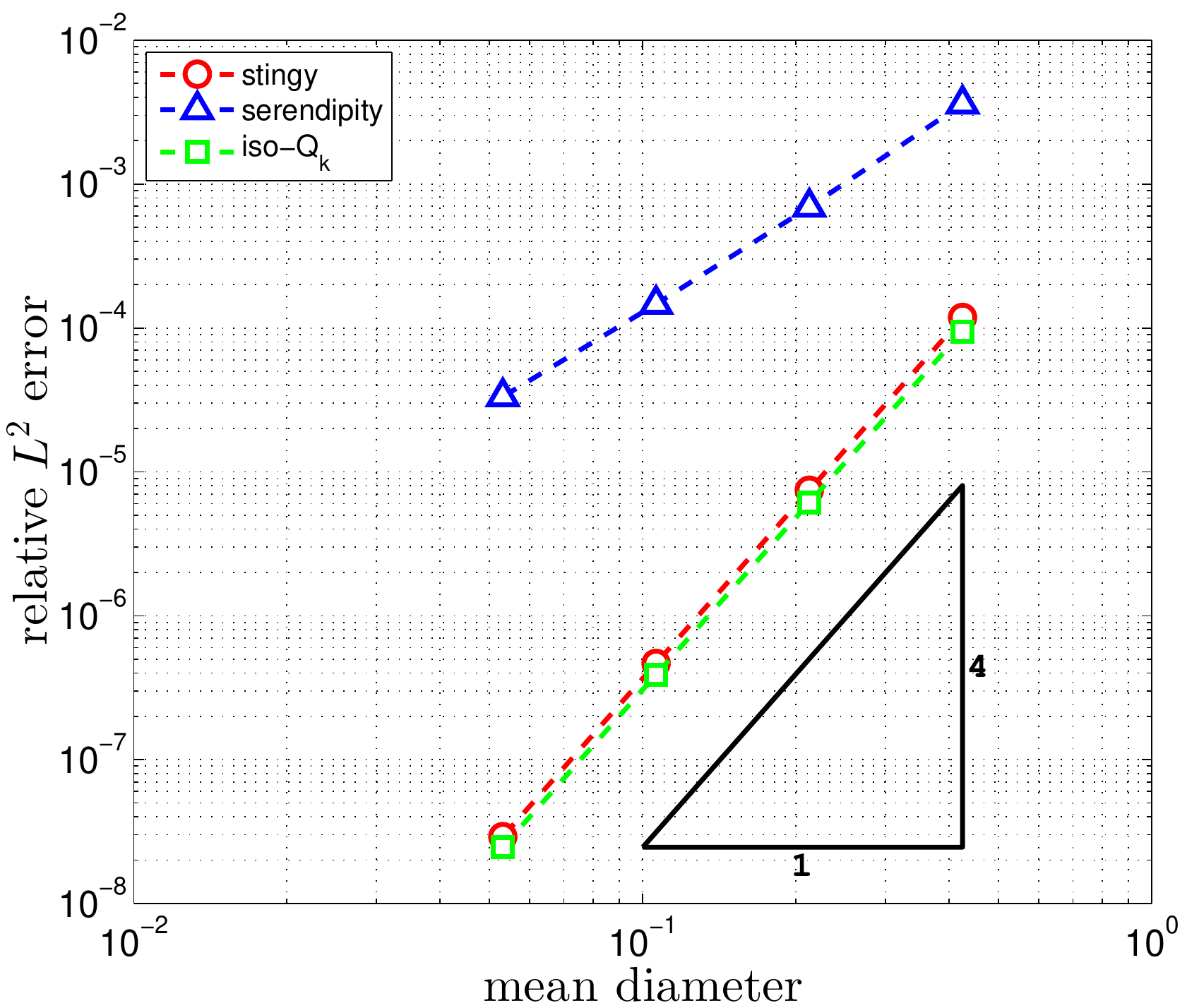}
  \end{center}
 \end{minipage}
\hfill
 \begin{minipage}[b]{0.49\textwidth}
  \begin{center}
	\include{./tableBBMR/table-boffi-k=3-QkSk}

	\bigskip
	\bigskip
	\end{center}
 \end{minipage}
  \caption{$k=3$, $L^2$ error for the trapezoidal meshes.}
 \label{fig:boffi-L2-k=3}
\hfill
\end{figure}

\begin{figure}[ht]
\hfill
 \begin{minipage}[b]{0.49\textwidth}
  \begin{center}
  \includegraphics[width=\textwidth,height=6.3cm]{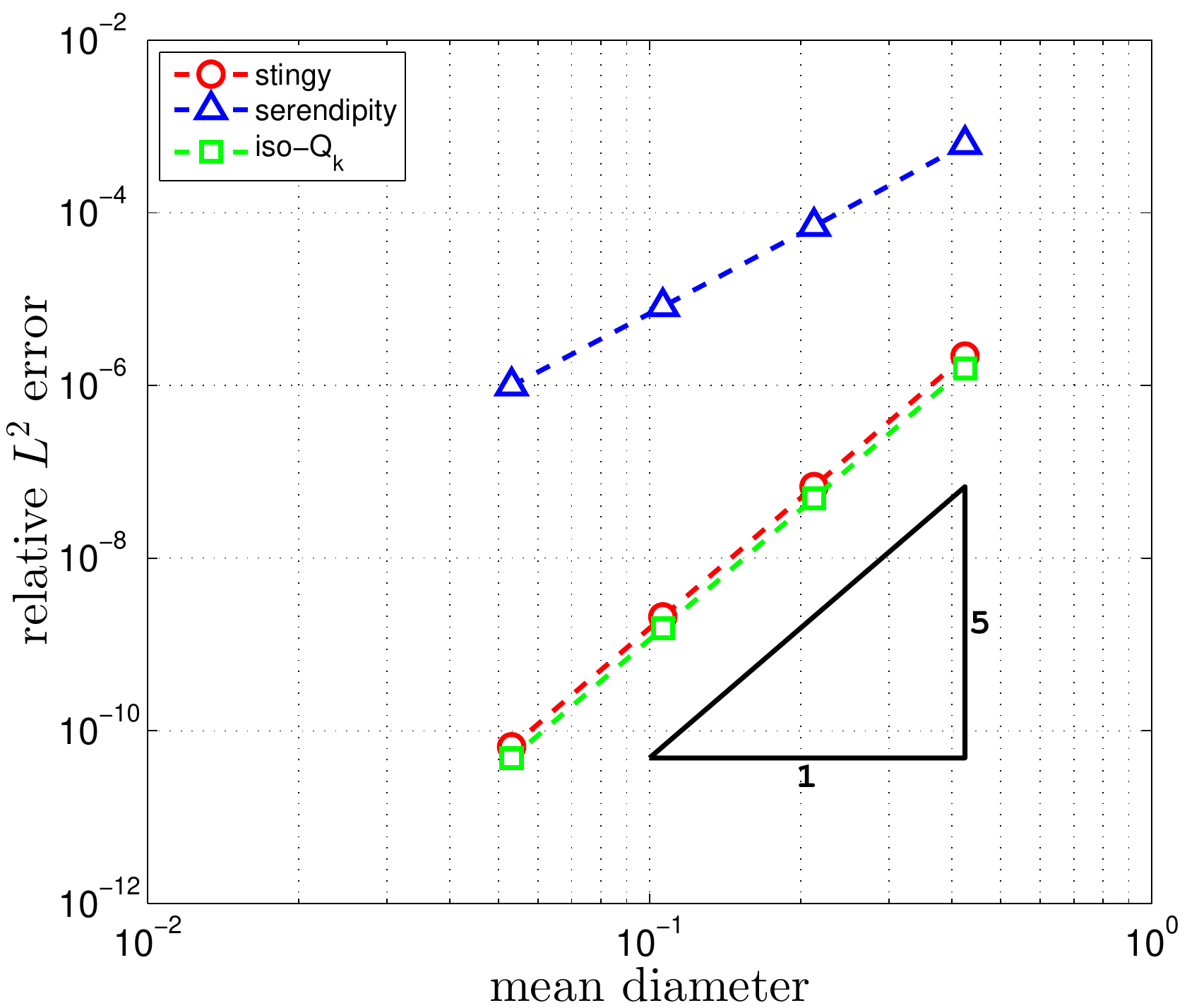}
  \end{center}
 \end{minipage}
\hfill
 \begin{minipage}[b]{0.49\textwidth}
  \begin{center}
	\include{./tableBBMR/table-boffi-k=4-QkSk}

	\bigskip
	\bigskip
	\end{center}
 \end{minipage}
  \caption{$k=4$, $L^2$ error for the trapezoidal meshes}
 \label{fig:boffi-L2-k=4}
\hfill
\end{figure}

\subsection{Polygonal meshes}

The polygonal meshes are made of 25, 100, 400 and 1600 polygons and have been
obtained starting with a random Voronoi mesh and then regularized by means
of Lloyd iterations. The 100 polygon mesh is shown in Fig. \ref{fig:lloyd}.

The equation that we solve is the same used for the numerical experiments in 
\cite{variable-primal}. We take
\begin{equation}
\diffp = \begin{pmatrix}y^2+1&-xy\\-xy&x^2+1\end{pmatrix},\quad
\bb=(x,y),\quad
\reaction = x^2+y^3,
\end{equation}
and right hand side $f$ and Dirichlet boundary condition $g$ defined
in such a way that the exact solution is
\begin{equation}
\pex(x,y) := x^2 y + \sin(2\pi x) \sin(2\pi y)+2.
\end{equation}
In Figs \ref{fig:lloyd-L2-k=2}, \ref{fig:lloyd-L2-k=3} and \ref{fig:lloyd-L2-k=4} 
we show the $L^2$ error for the ``stingy'' and the
``lazy'' strategies, and we compare them to the original VEM. 
Note that we have always taken $\eta_E$ equal to the number of edges
of the polygon $E$. 

In all cases we observe that the errors are very similar even if the
number of degrees of freedom is considerably different.

\begin{figure}[ht]
\hfill
 \begin{minipage}[b]{0.49\textwidth}
  \begin{center}
  \includegraphics[width=\textwidth,height=6.3cm]{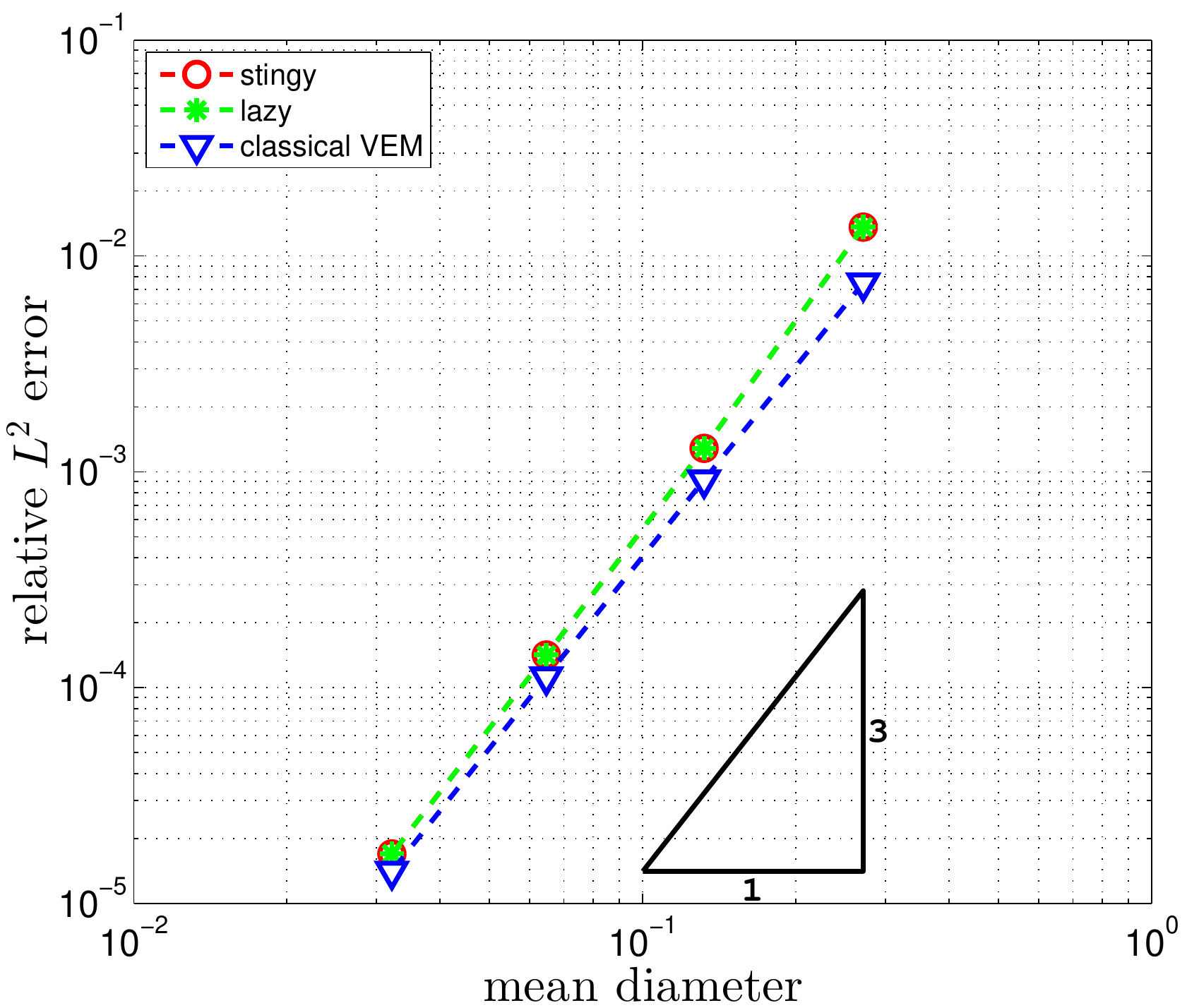}
  \end{center}
 \end{minipage}
\hfill
 \begin{minipage}[b]{0.49\textwidth}
  \begin{center}
	\include{./tableBBMR/table-lloyd-k=2}

	\bigskip
	\bigskip
	\end{center}
 \end{minipage}
  \caption{$k=2$, $L^2$ error for the Lloyd meshes}
 \label{fig:lloyd-L2-k=2}
\hfill
\end{figure}

\begin{figure}[!]
\hfill
 \begin{minipage}[b]{0.49\textwidth}
  \begin{center}
  \includegraphics[width=\textwidth,height=6.3cm]{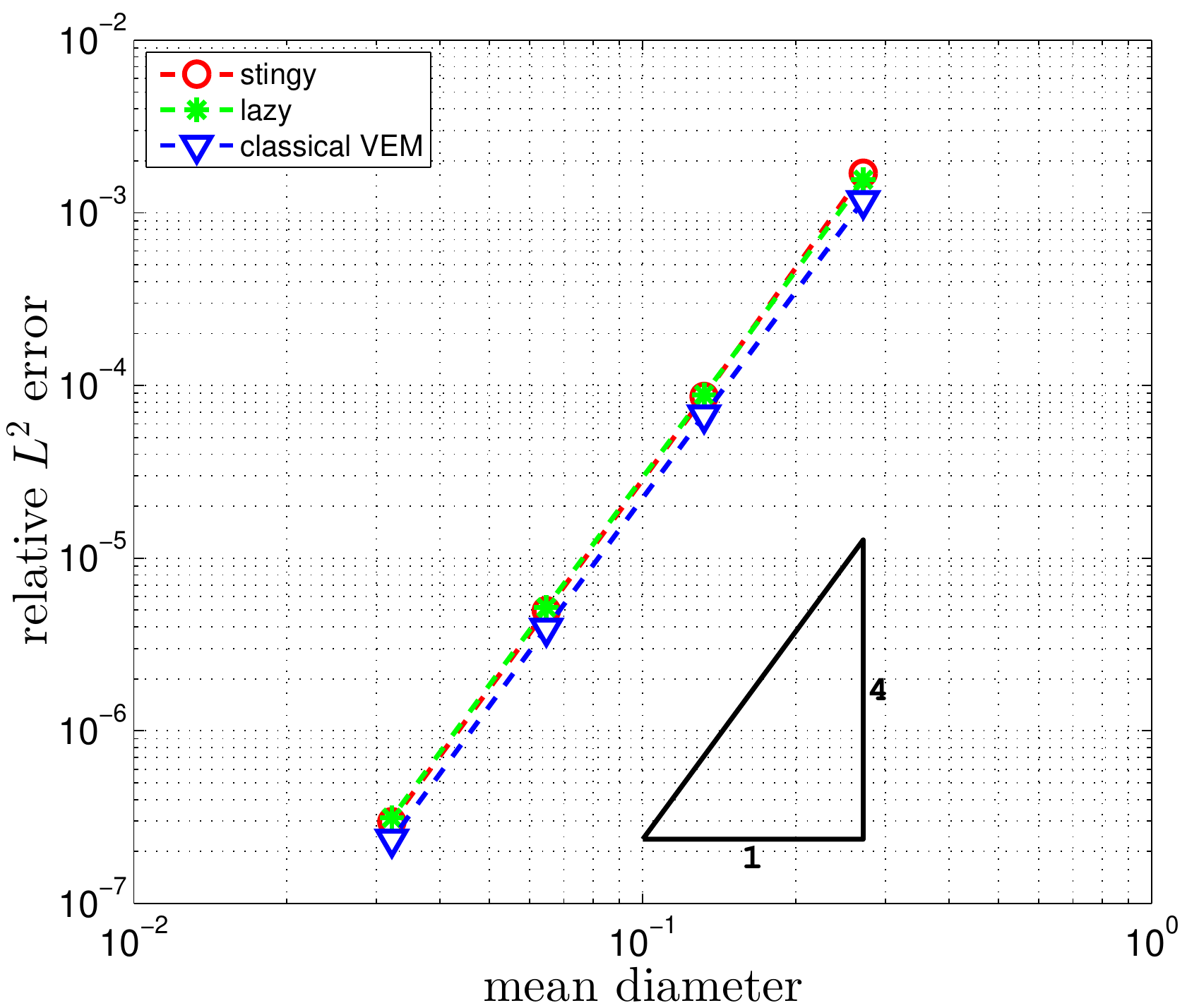}
  \end{center}
 \end{minipage}
\hfill
 \begin{minipage}[b]{0.49\textwidth}
  \begin{center}
	\include{./tableBBMR/table-lloyd-k=3}

	\bigskip
	\bigskip
	\end{center}
 \end{minipage}
  \caption{$k=3$, $L^2$ error for the Lloyd meshes}
 \label{fig:lloyd-L2-k=3}
\hfill
\end{figure}

\begin{figure}[!]
\hfill
 \begin{minipage}[b]{0.49\textwidth}
  \begin{center}
  \includegraphics[width=\textwidth,height=6.3cm]{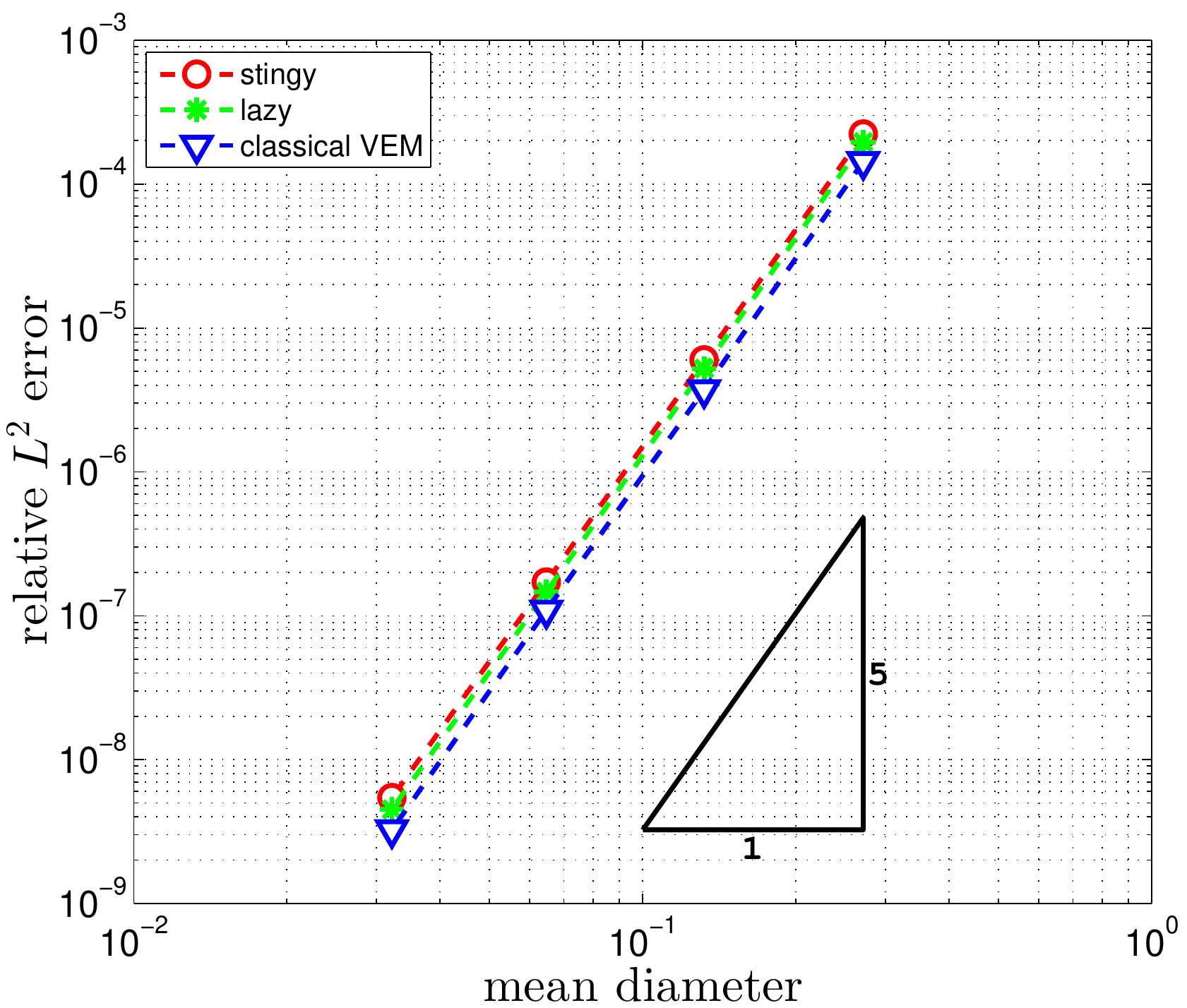}
  \end{center}
 \end{minipage}
\hfill
 \begin{minipage}[b]{0.49\textwidth}
  \begin{center}
	\include{./tableBBMR/table-lloyd-k=4}

	\bigskip
	\bigskip
	\end{center}
 \end{minipage}
  \caption{$k=4$, $L^2$ error for the Lloyd meshes}
 \label{fig:lloyd-L2-k=4}
\hfill
\end{figure}

\section{Conclusions}
Virtual Element Methods generalize Finite Elements from simple geometric shapes
(triangles, tetrahedrons, quadrilaterals, hexahedrons, etc.) to much more general shapes,
including several types of ``degenerations''. However, when restricted to simple
geometries they do not reproduce the traditional FEM, not even in the number of
degrees of freedom. For simplexes (in 2 or 3 dimensions), FEMs of order $k$ have a number of internal  degrees of freedom that is equal to $\pi_{k-d-1,d}$ (the dimension of the space of polynomials of degree $\le k-d-1$ in $d$ dimensions), while the number of internal d.o.f. of traditional VEMs is equal to $\pi_{k-d,d}$. On quadrilaterals and
hexahedrons traditional FEMs have ${\mathscr q}_{k-d,d}$ internal nodes (the dimension of the space of polynomials of degree $\le k-d$ {\it in each variable} in $d$ dimensions)
while VEMs do better with only $\pi_{k-d,d}$. Serendipity FEMs, however, can go
down to $\pi_{k-d-3,d}$, but they suffer dramatic losses of accuracy when the elements
are not parallelograms. Something quite similar also happens for hexahedrons.

Here we introduced a new family of VEMs that mimicks (in some sense) the Serendipity idea of FEM. These new elements reduce in a significant way the number of
internal degrees of freedom of traditional VEMs, without losing the good features of 
being able to deal with very general shapes and distortions. 

On triangles, the new
VEMs coincide now with Finite Elements, so that we don't gain anything apart from the conceptual satisfaction of equaling the ``competitors'' (in a friendly sense) where and when they are at their best.

On quads, however, the new VEMs can match the number of degrees of freedom of Serendipity FEM with {\it much} more generality in the geometry, and could therefore
become a competitor even for rather simple element shapes (as it is clearly shown by the numerical experiments of the previous section). On top of that,  they allow
extremely general geometries that are totally out of reach for Finite Elements.

We point out that in three dimensions our discussion  applies as well to the degrees of freedom that are {\it internal to the faces}, that therefore cannot be eliminated by a simple static condensation.

\bibliographystyle{amsplain}
\bibliography{general-bibliography}
\end{document}

%% file: tableBBMR/table-boffi-k=2-QkSk.tex
\begin{tabular}{ r | r | r | r |}
\cline{2-4}
& \multicolumn{3}{ c |  }{degrees of freedom} \\ 
\hline
\multicolumn{1}{ | c | }{\# el.}  & stingy & $\aleS_k$ & $\aleQ_k$ \\
\hline
\multicolumn{1}{ | c | }{16} & 65 & 65 & 81\\ 
\hline
\multicolumn{1}{ | c | }{64} & 225 & 225 & 289\\ 
\hline
\multicolumn{1}{ | c | }{256} & 833 & 833 & 1089\\ 
\hline
\multicolumn{1}{ | c | }{1024} & 3201 & 3201 & 4225\\ 
\hline
\end{tabular}

%% file: tableBBMR/table-boffi-k=3-QkSk.tex
\begin{tabular}{ r | r | r | r |}
\cline{2-4}
& \multicolumn{3}{ c |  }{degrees of freedom} \\ 
\hline
\multicolumn{1}{ | c | }{\# el.}  & stingy & $\aleS_k$ & $\aleQ_k$ \\
\hline
\multicolumn{1}{ | c | }{16} & 105 & 105 & 169\\ 
\hline
\multicolumn{1}{ | c | }{64} & 369 & 369 & 625\\ 
\hline
\multicolumn{1}{ | c | }{256} & 1377 & 1377 & 2401\\ 
\hline
\multicolumn{1}{ | c | }{1024} & 5313 & 5313 & 9409\\ 
\hline
\end{tabular}

%% file: tableBBMR/table-boffi-k=4-QkSk.tex
\begin{tabular}{ r | r | r | r |}
\cline{2-4}
& \multicolumn{3}{ c |  }{degrees of freedom} \\ 
\hline
\multicolumn{1}{ | c | }{\# el.}  & stingy & $\aleS_k$ & $\aleQ_k$ \\
\hline
\multicolumn{1}{ | c | }{16} & 161 & 161 & 289\\ 
\hline
\multicolumn{1}{ | c | }{64} & 577 & 577 & 1089\\ 
\hline
\multicolumn{1}{ | c | }{256} & 2177 & 2177 & 4225\\ 
\hline
\multicolumn{1}{ | c | }{1024} & 8449 & 8449 & 16641\\ 
\hline
\end{tabular}

%% file: tableBBMR/table-lloyd-k=2.tex
\begin{tabular}{ r | r | r | r |}
\cline{2-4}
& \multicolumn{3}{ c |  }{degrees of freedom} \\ 
\hline
\multicolumn{1}{ | c | }{\# el.}  & stingy & lazy & VEM \\
\hline
\multicolumn{1}{ | c | }{25} & 128 & 128 & 153\\ 
\hline
\multicolumn{1}{ | c | }{100} & 503 & 503 & 603\\ 
\hline
\multicolumn{1}{ | c | }{400} & 2003 & 2003 & 2403\\ 
\hline
\multicolumn{1}{ | c | }{1600} & 8003 & 8003 & 9603\\ 
\hline
\end{tabular}

%% file: tableBBMR/table-lloyd-k=3.tex
\begin{tabular}{ r | r | r | r |}
\cline{2-4}
& \multicolumn{3}{ c |  }{degrees of freedom} \\ 
\hline
\multicolumn{1}{ | c | }{\# el.}  & stingy & lazy & VEM \\
\hline
\multicolumn{1}{ | c | }{25} & 204 & 229 & 279\\ 
\hline
\multicolumn{1}{ | c | }{100} & 804 & 904 & 1104\\ 
\hline
\multicolumn{1}{ | c | }{400} & 3204 & 3604 & 4404\\ 
\hline
\multicolumn{1}{ | c | }{1600} & 12804 & 14404 & 17604\\ 
\hline
\end{tabular}

%% file: tableBBMR/table-lloyd-k=4.tex
\begin{tabular}{ r | r | r | r |}
\cline{2-4}
& \multicolumn{3}{ c |  }{degrees of freedom} \\ 
\hline
\multicolumn{1}{ | c | }{\# el.}  & stingy & lazy & VEM \\
\hline
\multicolumn{1}{ | c | }{25} & 284 & 355 & 430\\ 
\hline
\multicolumn{1}{ | c | }{100} & 1112 & 1405 & 1705\\ 
\hline
\multicolumn{1}{ | c | }{400} & 4408 & 5605 & 6805\\ 
\hline
\multicolumn{1}{ | c | }{1600} & 17614 & 22405 & 27205\\ 
\hline
\end{tabular}